\documentclass[11pt,letterpaper]{amsart}
\usepackage{amsmath,amsthm,latexsym,amssymb,mathrsfs}
\usepackage{mathabx}
\usepackage{mathtools}
\usepackage{dsfont}
\usepackage[english]{babel}
\usepackage[utf8]{inputenc}
\usepackage{color}
\usepackage{physics}
\usepackage{faktor}
\usepackage{overpic}
\usepackage{enumerate}
\usepackage[shortlabels]{enumitem}
\usepackage{pdflscape}
\usepackage{float}
\usepackage[alphabetic,nobysame]{amsrefs}
\usepackage{hyperref}
\usepackage[all,cmtip]{xy}
\usepackage{todonotes}
\usepackage{csquotes}
\usepackage{url}

\newtheoremstyle{mythm}% ⟨name⟩
{3pt}% ⟨Space above⟩
{3pt}% ⟨Space below⟩
{\itshape}% ⟨Body font⟩
{}% ⟨Indent amount⟩
{\bfseries}% ⟨Theorem head font⟩
{.}% ⟨Punctuation after theorem head⟩
{.5em}% ⟨Space after theorem head⟩
{\thmnote{#1 }#3}% ⟨Theorem head spec (can be left empty, meaning ‘normal’)⟩

\newtheorem{thm}{Theorem}[section]
\newtheorem{lem}[thm]{Lemma}
\newtheorem{cor}[thm]{Corollary}
\newtheorem{pro}[thm]{Proposition}
\newtheorem*{qn*}{Question}
\newtheorem*{lem*}{Lemma}

\newtheorem{thm*}{Theorem}
\newtheorem{pro*}[thm*]{Proposition}
\newtheorem{cor*}[thm*]{Corollary}

\newtheorem*{thm**}{Theorem}

\theoremstyle{mythm}

\theoremstyle{definition}
\newtheorem{dfn}[thm]{Definition}

\theoremstyle{remark}

\newtheorem*{claim*}{Claim}
\newtheorem*{fact*}{Fact}
\newtheorem{claiminner}{Claim}  % manually labelled claim
\newenvironment{claim}[1]{%
    \renewcommand\theclaiminner{#1}%
    \claiminner
}{\endclaiminner}

\newcommand{\ep}{
    \epsilon
}
\newcommand{\mc}[1]{
    \mathcal{#1}
}
\newcommand{\mb}[1]{
    \mathbb{#1}
}
\newcommand{\T}{
    \mc{T}
}

\begin{document}
    
\title{Length functions in Teichmüller and anti de Sitter geometry}

\author{Filippo Mazzoli}
\address{Department of Mathematics, University of Virginia, Charlottesville, USA}
\email{filippomazzoli@me.com}

\author{Gabriele Viaggi}
\address{Department of Mathematics, Heidelberg University, Heidelberg, Germany}
\email{gviaggi@mathi.uni-heidelberg.de}
    
\begin{abstract}
We establish a link between the behavior of length functions on Teichmüller space and the geometry of certain anti de Sitter 3-manifolds. As an application, we give new purely anti de Sitter proofs of results of Teichmüller theory such as (strict) convexity of length functions along shear paths and geometric bounds on their second variation along earthquakes. Along the way, we provide shear-bend coordinates for Mess' anti de Sitter 3-manifolds.
\end{abstract}
    
\maketitle

\section{Introduction}
The space $\T$ of hyperbolic metrics on a closed orientable surface $\Sigma$ of genus $g\ge 2$ up to isotopy, known as {\em Teichmüller space}, is an object that appears ubiquitously as a space of parameters but also as a geometric object. 

Comparing different hyperbolic metrics on $\Sigma$ according to various measurements of distortion endows $\T$ with a wealth of geometry. An example is the {\em Lipschitz distortion} which corresponds to the so-called {\em Thurston's asymmetric metric}. Thurston proves in \cite{T86} that given hyperbolic metrics $g_X,g_Y$ on $\Sigma$ we have
\[
\min_{\text{\rm $f$ homotopic to Id}}\{{\rm Lip}(f)\left|\;f:(\Sigma,g_X)\to(\Sigma,g_Y)\right.\}=\sup_{\gamma\in\pi_1(\Sigma)-\{1\}}{\frac{L_Y(\gamma)}{L_X(\gamma)}} 
\]
where $L_X(\gamma),L_Y(\gamma)$ are the lengths of the geodesic representatives of $\gamma$ with respect to $g_X,g_Y$. 

This phenomenon of expressing the measurement of distortion in terms of {\em length spectra} $L_Z(\bullet)$ is not exclusive of the Thurston metric, for example also the {\em Teichmüller} and {\em Weil-Petersson metrics} on $\T$ have this property. 

It is therefore important to understand how length functions behave on Teichmüller space. Often, this behavior is related to certain {\em geometric structures} on low dimensional manifolds. A celebrated example is the relation between {\em quasi-Fuchsian} hyperbolic $3$-manifolds and Teichmüller geodesics discovered by Minsky \cite{M93}. 

Following an analogy between quasi-Fuchsian 3-manifolds and the so-called {\em Mess 3-manifolds}, in this article we bring together:
\begin{itemize}
\item{3-dimensional anti de Sitter geometry.}
\item{Convexity of length functions along {\em shear paths} and {\em earthquakes}.} 
\end{itemize}

In particular, we use the global scale geometry of Mess manifolds to give a proof of (strict) convexity of length functions. Using the same bridge, we also develop geometric bounds for the second variation on those functions along {\em earthquakes}. Our methods are inspired from ideas in 3-dimensional hyperbolic geometry.

\subsection{Anti de sitter geometry}
Anti de Sitter geometry in dimension 3. is the geometry of $\mb{H}^{2,1}:={\rm PSL}_2(\mb{R})$ endowed with its natural pseudo-Riemannian metric of signature $(2,1)$. The link between Teichmüller theory and anti de Sitter 3-manifolds comes from the basic fact that the group of symmetries of this space is
\[
{\rm Isom}_0({\rm PSL}_2(\mb{R}))={\rm PSL}_2(\mb{R})\times{\rm PSL}_2(\mb{R})
\]
where $(A,B)\cdot X:=AXB^{-1}$ and, at the same time, ${\rm PSL}_2(\mb{R})={\rm Isom}^+(\mb{H}^2)$. A vast literature explores various aspects of this relation starting with the seminal work of Mess \cite{M07} (for a survey on the topic and recent developments see \cite{BS20}).

\subsection*{Mess representations}
Let $\Sigma$ be a closed orientable surface of genus $g\ge 2$ that we fix once and for all. We denote by $\Gamma:=\pi_1(\Sigma)$ its fundamental group. 

We realize the Teichmüller space $\T$ of hyperbolic metrics on $\Sigma$ up to isotopy as a component of the representation space
\[
\T\subset{\rm Hom}(\Gamma,{\rm PSL}_2(\mb{R}))/{\rm PSL}_2(\mb{R})
\]
by associating to each hyperbolic structure $X$ its holonomy representation $\rho_X:\Gamma\to{\rm PSL}_2(\mb{R})$.

Given $X,Y\in\T$ we can consider the corresponding {\em Mess representation} 
\[
\rho_{X,Y}=(\rho_X,\rho_Y):\Gamma\to{\rm PSL}_2(\mb{R})\times{\rm PSL}_2(\mb{R}).
\]
The group $\rho_{X,Y}(\Gamma)$ acts on $\mb{H}^{2,1}$ {\em convex cocompactly}, meaning that:
\begin{itemize}
\item{There is an equivariant boundary map $$\xi:\partial\Gamma\to\partial{\rm PSL}_2(\mb{R})=\mb{P}\{A\in M_2(\mb{R})\left|{\rank }(A)=1\right.\} ,$$
whose image $\xi(\partial\Gamma)=\Lambda_{X,Y}$ has the property that for every $a,b,c\in\partial\Gamma$ the subspace $\mb{P}\{{\rm Span}\{\xi(a),\xi(b),\xi(c)\}\}\cap{\rm PSL}_2(\mb{R})$ is a {\em spacelike plane}, that is, it is isometric to $\mb{H}^2$.}
\item{There is a canonical $\rho_{X,Y}(\Gamma)$-invariant properly convex open subset $\Omega_{X,Y}\subset{\rm PSL}_2(\mb{R})$ on which the action is properly discontinuous.}
\item{We have $\partial\Omega_{X,Y}\cap\partial{\rm PSL}_2(\mb{R})=\Lambda_{X,Y}$ and the group $\rho_{X,Y}(\Gamma)$ acts cocompactly on the convex hull $\mc{CH}_{X,Y}\subset\Omega_{X,Y}$ of $\Lambda_{X,Y}$.}
\end{itemize}

In order to study the geometry of Mess representations, we will use {\em laminations} and {\em pleated surfaces} as we introduced in \cite{MV22a}. Let us briefly recall the construction.

\subsection*{Laminations and pleated surfaces}
A geodesic lamination on a hyperbolic surface $X$ is a $\rho_X(\Gamma)$-invariant closed subset $\lambda\subset\mb{H}^2$ that can be decomposed as a disjoint union of complete geodesics, the {\em leaves} of the lamination. The connected components of $\mb{H}^2-\lambda$ are ideal polygons, the {\em plaques} of the lamination. The lamination is called {\em maximal}, if all the plaques are ideal triangles. Conveniently, the data of a geodesic lamination can be encoded, by recording the endpoints of the leaves, as a $\Gamma$-invariant closed subset of the space of geodesics
\[
\{(x,y)\in\partial\Gamma\times\partial\Gamma\left|\;x\neq y\right.\}/(x,y)\sim(y,x).
\]
This is the point of view that we adopt. 

The boundary map $\xi:\partial\Gamma\to\Lambda_{X,Y}$ and the property of the curve $\Lambda_{X,Y}$ allow us to associate with every maximal lamination $\lambda$ a {\em geometric realization}
\[
{\hat \lambda}:=\bigcup_{(a,b)\in\lambda}{[\xi(a),\xi(b)]}\subset\mc{CH}_{X,Y}
\]
and a {\em pleated set}
\[
{\hat S_\lambda}:={\hat \lambda}\cup\bigcup_{\Delta(a,b,c)\subset\mb{H}^2-\lambda}{\Delta(\xi(a),\xi(b),\xi(c))}\subset\mc{CH}_{X,Y}.
\]
Here $[\xi(a),\xi(b)]$ denotes the {\em spacelike} geodesic with endpoints $\xi(a),\xi(b)$ while $\Delta(\xi(a),\xi(b),\xi(c))$ is the ideal {\em spacelike} triangle contained in the spacelike plane $\mb{P}\{{\rm Span}\{\xi(a),\xi(b),\xi(c)\}\}\cap{\rm PSL}_2(\mb{R})$ with vertices $\xi(a),\xi(b),\xi(c)$.

We have the following structural result: 

\begin{thm**}[Theorems A, B, and C of \cite{MV22a}]
Let $\rho_{X,Y}$ be a Mess representation. Consider a maximal lamination $\lambda\subset\Sigma$. Let ${\hat S_\lambda}\subset\mc{CH}_{X,Y}\subset\Omega_{X,Y}$ be the corresponding {\rm pleated set}. Then:
\begin{enumerate}
\item{${\hat S}_\lambda$ is an {\rm acausal} Lipschitz disk with boundary $\Lambda_{X,Y}$. For every pair points $x,y\in{\hat S}_\lambda$ the geodesic $[x,y]$ joining them is spacelike. In particular ${\hat S}_\lambda$ has a pseudo-metric $d_{\mb{H}^{2,1}}(x,y):=\ell[x,y]$.}
\item{There is an {\rm intrinsic hyperbolic structure} $Z_\lambda\in\T$ associated to $\lambda$ with holonomy $\rho_\lambda:\Gamma\to{\rm PSL}_2(\mb{R})$. For every $\mu\in\mc{ML}_\lambda=\{\mu\in\mc{ML}\left|\;{\rm support}(\mu)\subset\lambda\right.\}$ we have $L_{\rho_\lambda}(\mu)=L_{\rho}(\mu)$.}
\item{There exists a $(\rho_{X,Y}-\rho_\lambda)$-equivariant homeomorphism ${\hat f}:{\hat S}_\lambda\to\mb{H}^2$ which is 1-Lipschitz in the sense that $d_{\mb{H}^{2,1}}(x,y)\ge d_{\mb{H}^2}({\hat f}(x),{\hat f}(y))$ and is totally geodesic on each leaf and plaque.}
\end{enumerate}

Furthermore, we have
\[
L_{\rho_\lambda}(\gamma)\le L_{\rho_{X,Y}}(\gamma)
\]
for every $\gamma\in\Gamma-\{1\}$, with strict inequality if and only if $\gamma$ intersects the {\rm bending locus} of $\hat{S}_\lambda$.
\end{thm**}

Mess \cite{M07}, inspired by work of Thurston (Chapter 8 of \cite{ThNotes}), observes that $\partial^\pm\mc{CH}_{X,Y}$ is the pleated set ${\hat S}_{\lambda^\pm}$ of a lamination $\lambda^\pm$ and that measuring the total turning angle along paths $\alpha:I\to\partial^\pm\mc{CH}_{X,Y}$ endows $\lambda^\pm$ with a natural transverse measure, the {\em bending measure}. Then he shows that the surfaces $X,Y$ and $Z_{\lambda^+},Z_{\lambda^-}$ are related by the following diagram
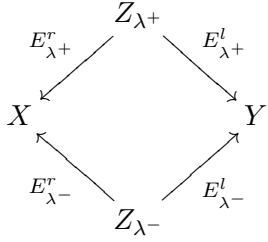
\[
\xymatrix{
&Z_{\lambda^+}\ar[dl]_{E_{\lambda^+}^r}\ar[dr]^{E_{\lambda^+}^l} &\\
X & &Y\\
&Z_{\lambda^-}\ar[ul]^{E_{\lambda^-}^r}\ar[ur]_{E_{\lambda^-}^l} &
}
\]
where $E_{\lambda^+}^l,E_{\lambda^-}^l,E_{\lambda^+}^r,E_{\lambda^-}^r$ are the {\em left} and {\em right earthquakes} induced by the measured laminations $\lambda^+,\lambda^-$.  

Recall that by work of Bonahon \cite{Bo96} and Thurston \cite{T86}, for every maximal geodesic lamination $\lambda$ of $\Sigma$, the Teichmüller space $\T$ can be realized as an open convex cone in a finite dimensional $\mb{R}$-vector space $\mc{H}(\lambda;\mb{R})$ via the so-called {\em shear coordinates} $\sigma_\lambda:\T\to\mc{H}(\lambda;\mb{R})$. Generalizing Mess, we prove:

\begin{thm*}
\label{thm:pleated surfaces ads}
Let $\rho_{X,Y}$ be a Mess representation. Consider a maximal lamination $\lambda\subset\Sigma$. Let $S_\lambda={\hat S_\lambda}/\rho_{X,Y}(\Gamma)$ be the corresponding pleated surface. Then, in {\rm shear coordinates} $\T\subset\mc{H}(\lambda;\mb{R})$ for $\lambda$ we have:
\begin{enumerate}
\item{The {\rm intrinsic hyperbolic structure} $Z_\lambda$ of $S_\lambda$ satisfies 
\[
\sigma_\lambda(Z_\lambda)=\frac{\sigma_\lambda(X)+\sigma_\lambda(Y)}{2}.
\]
}
\item{The {\rm intrinsic bending cocycle} $\beta_\lambda$ of $S_\lambda$ satisfies 
\[
\beta_\lambda=\frac{\sigma_\lambda(X)-\sigma_\lambda(Y)}{2}.
\]
}
\end{enumerate}
\end{thm*}

\subsection*{Length functions}

We now come to the main novelty of this article, namely, the anti de Sitter perspective on length functions in Teichmüller theory. 

Let us first recall the following: For every element $\gamma\in\Gamma-\{1\}$ the isometry $\rho_{X,Y}(\gamma)$ has a unique pair of invariant spacelike lines: The {\em axis} $\ell\subset\mc{CH}_{X,Y}$ on which it acts by translations by $L_\rho(\gamma)=(L_X(\gamma)+L_Y(\gamma))/2$ and the dual axis $\ell^*\subset\mb{H}^{2,1}-\Omega_{X,Y}$ on which it acts by translations by $\theta_\rho(\gamma)=(L_X(\gamma)-L_Y(\gamma))/2$. 

We prove:

\begin{thm*}
\label{thm:length}
Let $\rho_{X,Y}$ be a Mess representation and let $\lambda$ be a maximal lamination. Denote by $Z_\lambda\in\T$ the intrinsic hyperbolic structure on ${\hat S}_\lambda/\rho_{X,Y}(\Gamma)$ where ${\hat S}_\lambda\subset\mc{CH}_{X,Y}$ is the pleated set associated with $\lambda$, and consider $\gamma\in\Gamma-\{1\}$ a non-trivial element whose image $\rho_{X,Y}(\gamma)$ has axis $\ell\subset\mc{CH}_{X,Y}$.
\begin{enumerate}[(a)]
\item{Let $\delta$ be the maximal timelike distance of $\ell$ from ${\hat S}_\lambda$. Then:
\[
\cosh(L_{Z_\lambda}(\gamma))\le\cos(\delta)^2\cosh(L_\rho(\gamma))+\sin(\delta)^2\cosh(\theta_\rho(\gamma)).
\]
}
\item{Let $\delta^\pm$ be the maximal timelike distance of $\ell$ from $\lambda^\pm$. Then:
\[
\cosh(i(\lambda^\pm,\gamma))\le\sin(\delta^\pm)^2\cosh(L_\rho(\gamma))+\cos(\delta^\pm)^2\cosh(\theta_\rho(\gamma)).
\]
Here $i(\bullet,\bullet)$ is the {\rm geometric intersection} form.}
\end{enumerate}
\end{thm*}

When combined, the previous results (Theorem \ref{thm:pleated surfaces ads} and Theorem \ref{thm:length}) give a purely anti de Sitter proof of (strict) convexity of length functions in shear coordinates, recovering simultaneously results of Bestvina, Bromberg, Fujiwara, and Souto \cite{BBFS13}, and Théret \cite{The14}:

\begin{thm*}
\label{thm:length convex}
Let $\lambda\subset\Sigma$ be a maximal lamination. The following holds:
\begin{enumerate}[(a)]
\item{Let $\gamma\in\Gamma-\{1\}$ be a non-trivial loop. The length function $L_\gamma:\T\subset\mc{H}(\lambda;\mb{R})\to(0,\infty)$ is convex. Moreover, convexity is strict if $\gamma$ intersects essentially every leaf of $\lambda$.}
\item{Let $\gamma\in\mc{ML}$ be a measured lamination. The length function $L_\gamma:\T\subset\mc{H}(\lambda;\mb{R})\to(0,\infty)$ is convex. Furthermore, convexity is strict if the support of $\gamma$ intersects transversely each leaf of $\lambda$.}
\end{enumerate}
\end{thm*}  

Note that (b) does not imply (a): In (a) the loop $\gamma$ does not necessarily represent a simple curve.

In the case of {\em earthquakes}, the geometry of Mess $3$-manifolds allows us to get the following infinitesimal geometric bound. We should mention that these bounds can also be deduced from work of Kerckhoff \cite{K83} and Wolpert \cite{W86} respectively.

\begin{thm*}
\label{thm:earthquakes}
Let $\lambda\in\mc{ML}$ be a measured lamination. Let $E_\lambda:[a,b]\to\T$ be an earthquake path driven by $\lambda$. Set $L_{\gamma}(t):=\ell_\gamma(E_\lambda(t))$. Then: For every $\gamma\in\Gamma-\{1\}$ we have: 
\[
{\ddot L_\gamma}\ge\frac{1}{\sinh(L_\gamma)}\left|{\dot L}_\gamma\right|\left(i(\gamma,\lambda)-\left|{\dot L}_\gamma\right|\right).
\]
\end{thm*}

Let us point out that, by Kerckhoff's formula for the first variation \cite{K83}, we always have $|{\dot L}_\gamma|\le i(\gamma,\lambda)$ with strict inequality if $\gamma$ intersects $\lambda$ essentially.

\subsection*{Anti de Sitter proofs}
We now briefly discuss the main new ideas and ingredients that go into the anti de Sitter proofs.

\subsubsection*{Theorem \ref{thm:length}}
The idea is that as we move a closed geodesic $\gamma\subset M_{X,Y}$ orthogonally along timelike directions, the length shrinks. Heuristically speaking: Every closed geodesic $\gamma\subset M_{X,Y}$ is the core of an (immersed) anti de Sitter annulus $A_\gamma\subset M_{X,Y}$ whose intrinsic metric has the form ${\rm d}s^2=-{\rm d}t^2+\sin(t)^2{\rm d}\ell^2$. Hence, the length of $\gamma(s)=(0,s)$ (in $(t,\ell)$ coordinates) contracts as we move it away from the core $\{t=0\}$ along orthogonal timelike directions. In the proof of the theorem we make some aspects of this picture precise. In particular, we understand how various avatars of $A_\gamma$ intersect the pleated surfaces ${\hat S}_\lambda/\rho_{X,Y}(\Gamma)$ and $\partial^\pm\mc{CC}(M_{X,Y})=\partial^\pm\mc{CH}_{X,Y}/\rho_{X,Y}(\Gamma)$.

\subsubsection*{Theorem \ref{thm:length convex}} 
(Strict) convexity is equivalent to the (strict) inequality in
\[
L_\gamma\left( Z_\lambda \right)\le\frac{L_\gamma(X)+L_\gamma(Y)}{2}
\]
for every $X \neq Y \in \T$, where $\sigma_\lambda(Z_\lambda) = \frac{\sigma_\lambda(X) + \sigma_\lambda(Y)}{2} \in \mathcal{H}(\lambda;\mathbb{R})$. We note that the right hand side is $L_\rho(\gamma)$ for $\rho_{X,Y}$ and the left hand side is, by Theorem \ref{thm:pleated surfaces ads}, $L_{Z_\lambda}(\gamma)$ where $Z_\lambda$ is the hyperbolic structure on the pleated surface ${\hat S}_\lambda/\rho(\gamma)$ associated with $\lambda$ and $\rho$. The inequality is then a consequence of part (a) of Theorem \ref{thm:length}. The inequality is not strict exactly when the maximal distance $\delta$ between $\ell$ and $\hat{S}_\lambda$ vanishes or, in other words, when $\ell\subset{\hat S}_\lambda$. This is possible if and only if $\gamma$ does not intersect the {\em bending locus}. 

The proof for laminations requires a significantly more refined argument based on the following heuristic principle: Every time $\ell$ passes at timelike distance $\delta>0$ from ${\hat S}_\lambda$ it creates a gap of size $\kappa>0$ between $L_{Z_\lambda}(\gamma)$ and $L_\rho(\gamma)$.

\subsubsection*{Theorem \ref{thm:earthquakes}} 
The idea is to analyze the geometry of the representations $\rho_t:=\rho_{Z_{-t},Z_t}$ where $Z_t=E^l_\lambda(t)$ as $t\to 0$. Notice that,  by Theorem \ref{thm:pleated surfaces ads}, the bending lamination on $\partial^+\mc{CH}_{Z_{-t},Z_t}$ is $\lambda_t^+=t\lambda$ and its associated hyperbolic structure is constant $Z_t^+=Z$. The main tool is again Theorem \ref{thm:length}: Combining the inequalities of part (a) and part (b) of Theorem \ref{thm:length} we have
\[
\cosh(t\cdot i(\lambda^+,\gamma))-\cosh(\theta_{\rho_t}(\gamma))\le\cosh(L_{\rho_t}(\gamma))-\cosh(L_Z(\gamma)).
\]
The conclusion follows from basic analysis, essentially the mean value theorem $\cosh(x)-\cosh(y)=\sinh(\xi)(x-y)$ where $\xi\in[x,y]$ and the fact that $(f(-t)+f(t)-2f(0))/t^2\to{\ddot f}$ which we apply to $(L_{\rho_t}(\gamma)-L_Z(\gamma))/t^2=(L_{Z_{-t}}(\gamma)+L_{Z_t}(\gamma)-2L_Z(\gamma))/2t^2\to{\ddot L}_\gamma/2$.

\subsection*{Shear-bend parametrization}
As an application of our computations on the intrinsic hyperbolic structure and intrinsic bending of a non-convex pleated surface, we also obtain a shear-bend parametrization of the space of Mess 3-manifolds in the spirit of Bonahon's work \cite{Bo96}: Consider the space of Mess representations
\[
\mc{MR}:=\T\times\T\subset{\rm Hom}(\Gamma,{\rm PSL}_2(\mb{B}))/{\rm PSL}_2(\mb{B}),
\]
where $\mb{B}:=\mb{R}[\tau]/(\tau^2-1)=\mb{R}\oplus\tau\mb{R}$ denotes the ring of {\em para-complex numbers}. Let $\mc{H}(\lambda;\mb{B})$ be the finite dimensional $\mb{B}$-module of {\em transverse cocycles} for $\lambda$ with values in $\mb{B}$ as introduced by Bonahon \cite{Bo96}. Notice that there are natural identifications ${\rm PSL}_2(\mb{B})={\rm PSL}_2(\mb{R})\times{\rm PSL}_2(\mb{R})$ and $\mc{H}(\lambda;\mb{B})=\mc{H}(\lambda;\mb{R})\oplus\tau\mc{H}(\lambda;\mb{R})$. We have:

\begin{thm*}
\label{thm:shear-bend ads}
 Let $\lambda\subset\Sigma$ be a maximal lamination. Then:
\begin{enumerate}
\item{The map
\[
\begin{array}{c}
\Phi:\mc{MR}\rightarrow\mc{H}(\lambda;\mb{B})\\
\rho\rightarrow\sigma_\rho:=Z_\lambda+\tau\beta_\lambda
\end{array}
\]
that associates to $\rho$ the {\rm shear-bend cocycle} of the unique pleated surface $S_\lambda={\hat S}_\lambda/\rho(\Gamma)$ associated with $\lambda$ is an analytic para-complex embedding.}
\item{If $\omega_{{\rm Th}}(\bullet,\bullet)$ denotes the {\rm Thurston's symplectic form} on $\mc{H}(\lambda;\mb{B})$, then
\[
\omega_{{\rm Th}}^\mb{B}(\sigma_\rho,\alpha)=L_\rho(\alpha)+\tau \, \theta_\rho(\alpha)
\]
for every measured lamination $\alpha\in\mc{ML}_\lambda=\{\alpha\in\mc{ML}\left|\text{ \rm support}(\alpha)\subset\lambda\right.\}$ and every $\rho\in\mc{MR}$.}
\item{The image of the embedding is given by 
\begin{align*}
\Phi(\mc{MR}) &=\left\{\sigma+\tau\beta\in\mc{H}(\lambda;\mb{B})\left|\;\sigma+\beta,\sigma-\beta\in\T\subset\mc{H}(\lambda;\mb{R})\right.\right\}\\
&=\left\{\sigma+\tau\beta\in\mc{H}(\lambda;\mb{B})\left|\;|\omega_{{\rm Th}}^\mb{B}(\sigma+\tau\beta,\bullet)|_{\mb{B}}^2>0\text{ \rm on }\mc{ML}_\lambda\right.\right\}.
\end{align*}
Here $|x+\tau y|_\mb{B}^2=x^2-y^2$ is the para-complex norm.} 
\item{The pull-back of $\omega_{{\rm Th}}$ to $\mc{MR}=\T\times\T$ coincides with
\[
\Phi^*\omega_{{\rm Th}}=c\cdot(\omega_{{\rm WP}}\oplus \omega_{{\rm WP}})
\]
where $\omega_{{\rm WP}}(\bullet,\bullet)$ is the Weil-Petersson symplectic form.}
\end{enumerate}
\end{thm*}

\subsection*{Structure of the article}
The paper is organized as follows:
\begin{itemize}
\item{In Section \ref{sec:sec3} we recall some basic facts in Teichmüller theory and anti de Sitter 3-dimensional geometry.}
\item{In Section \ref{sec:sec7} we introduce Mess representations and pleated surfaces and recall some of their properties. }
\item{In Section \ref{sec:sec4} we compute the intrinsic shear-bend cocycles of pleated surfaces and prove Theorems \ref{thm:pleated surfaces ads} and \ref{thm:shear-bend ads}. }
\item{In Section \ref{sec:sec5} we study the behavior of length functions for Mess representations and prove Theorem \ref{thm:length}. }
\item{In Section \ref{sec:sec6} we discuss the purely anti de Sitter proofs of Theorems \ref{thm:length convex} and \ref{thm:earthquakes}.}
\end{itemize}

\subsection*{Acknowledgements}
We thank Sara Maloni, Beatrice Pozzetti, and Andrea Seppi for useful discussions and generous feedback on the article.

Gabriele gratefully acknowledges the support of the DFG 427903332.

\section{Teichmüller and anti de Sitter space}
\label{sec:sec3}

In this section we recall the amount of basic Teichmüller theory and anti de Sitter 3-dimensional geometry that we will need in the next sections.

\subsection{Teichmüller theory}
We start with hyperbolic surfaces and (measured) geodesic laminations. 

\subsubsection{Hyperbolic surfaces}
We fix once and for all a closed oriented surface $\Sigma$ of genus $g\ge 2$ and denote by $\Gamma:=\pi_1(\Sigma)$ its fundamental group.

\begin{dfn}[Hyperbolic Structures]
A {\em marked hyperbolic structure} on $\Sigma$ is a quotient $\mb{H}^2/\rho_X(\Gamma)$ of the hyperbolic plane $\mb{H}^2$ by the image of a faithful and discrete representation $\rho_X:\Gamma\to{\rm PSL}_2(\mb{R})$, the {\em holonomy} of the structure. Two marked hyperbolic structures $X,X'$ on $\Sigma$ are equivalent if their holonomies $\rho_X,\rho_{X'}$ are conjugate.
\end{dfn}

\begin{dfn}[Teichmüller Space]
The {\em Teichmüller space} of $\Sigma$, denoted by $\T$, is the space of equivalence classes of marked hyperbolic structures on $\Sigma$. It can be realized as a connected component of the space 
\[
\T\subset{\rm Hom}(\Gamma,{\rm PSL}_2(\mb{R}))/{\rm PSL}_2(\mb{R})
\]
where ${\rm PSL}_2(\mb{R})$ acts on the space of representations by conjugation.
\end{dfn}

\subsubsection{Geodesic laminations}
To study the geometry of hyperbolic surfaces it is quite useful to look at the behavior of their geodesic laminations which are 1-dimensional objects generalizing simple closed geodesics. 

\begin{dfn}[Space of Geodesics]
The space of (unoriented) geodesics on $\mb{H}^2$ is naturally identified with the set of pairs of endpoints
\[
\mc{G}:=\{(x,y)\in\mb{RP}^1\times\mb{RP}^1\left|x\neq y\right\}/(x,y)\sim(y,x)
\]
where $(x,y)$ corresponds to the line with endpoints $x,y$.
\end{dfn}

\begin{dfn}[Geodesic Lamination]
Let $X=\mb{H}^2/\rho_X(\Gamma)$ be a hyperbolic surface. A {\em geodesic lamination} on $X$ is a $\rho_X(\Gamma)$-invariant closed subset $\lambda\subset\mb{H}^2$ which can be expressed as a disjoint union of complete geodesics, the {\em leaves} of the lamination. The complementary regions $\mb{H}^2-\lambda$ are ideal polygons (with possibly infinitely many sides) and are called the {\em plaques} of $\lambda$. The geodesic lamination $\lambda$ is {\em maximal} if all its plaques are ideal triangles. A geodesic lamination on $X$ is completely determined by the endpoints on $\mb{RP}^1$ of the leaves which form a closed $\rho_X(\Gamma)$-invariant subset of $\mc{G}$. We denote by $\mc{GL}$ the space of geodesic laminations and by $\mc{GL}_m$ the subspace consisting of maximal ones.
\end{dfn}

For more details, we address the reader to Chapter I.4 of \cite{CEG}.

\subsubsection{Currents and measured laminations}
Both Teichmüller space and measured laminations can be seen inside the space of geodesic currents as introduced by Bonahon (see \cite{Bo88}). This framework is well-suited to study length functions thanks to presence of a natural geometric intersection form as we now explain.  

\begin{dfn}[Geodesic Current]
Let $X=\mb{H}^2/\rho_X(\Gamma)$ be a hyperbolic surface. A {\em geodesic current} on $X$ is a $\rho_X(\Gamma)$-invariant locally finite Borel measure on $\mc{G}$. We denote by $\mc{C}$ the space of geodesic currents. 
\end{dfn}

\begin{dfn}[Closed Geodesics]
A basic example of geodesic current is the one associated to a (free homotopy class) of a loop $\gamma\in\Gamma-\{1\}$. It is defined as $\delta_\gamma:=\sum_{[\alpha]\in\Gamma/\langle\gamma\rangle}{\delta_{\ell_\alpha}}$ where $\ell_\alpha$ is the axis of $\rho_X(\alpha)$ and $\delta_\ell$ is the Dirac mass on the point $\ell\in\mc{G}$.
\end{dfn}

\begin{dfn}[Geometric Intersection]
On $\mc{C}$ there is a natural {\em intersection form} $i(\bullet,\bullet)$ defined as follows: Let $\alpha,\beta\in\mc{C}$ be geodesic currents. Consider the space of intersecting geodesics $\mc{I}:=\{(\ell,\ell')\in\mc{G}\times\mc{G}\left|\ell\cap\ell'\neq\emptyset\right.\}$. The group $\rho_X(\Gamma)$ acts properly discontinuously and freely on $\mc{I}$. By invariance, the measure $\alpha\times\beta$ on $\mc{I}$ descends to a Borel measure on $\mc{I}/\rho_X(\Gamma)$. Define $i(\alpha,\beta):=\alpha\times\beta(\mc{I}/\rho_X(\Gamma))$. An crucial property of the geometric intersection form $i(\alpha,\beta)$ is that it is continuous in $\alpha,\beta$.
\end{dfn}

\begin{dfn}[Measured Lamination]
Let $X=\mb{H}^2/\rho_X(\Gamma)$ be a hyperbolic surface. A {\em measured lamination} on $X$ is a geodesic current $\lambda\in\mc{C}$ with $i(\lambda,\lambda)=0$. We denote by $\mc{ML}$ the space of measured laminations.

The {\em support} of a measured lamination ${\rm support}(\lambda)$ is a geodesic lamination (see \cite{Bo88}). We denote by $\mc{ML}_\lambda:=\{\mu\in\mc{ML}\left|\;{\rm support}(\mu)\subset\lambda\right.\}$ the space of measured laminations whose support is contained in $\lambda$.
\end{dfn}

\subsubsection{Length functions}
Every hyperbolic surface $X$ has a (marked) length spectrum $\{L_X(\gamma)\}_{\gamma\in\Gamma-\{1\}}$ given by the lengths of its closed geodesics. Conveniently, Bonahon \cite{Bo88} proves that the length function $L_X(\bullet)$ extends continuously to geodesic currents as follows:

\begin{dfn}[Liouville Current]
The {\em Liouville current} $\mc{L}$ on $\mc{G}$ is the ${\rm PSL}_2(\mb{R})$-invariant Borel measure on $\mc{G}$ defined by 
\[
\mc{L}([a,b]\times[c,d]):=\beta^\mb{R}(a,b,c,d).
\]
on boxes $[a,b]\times[c,d]$ with $[a,b]\cap[c,d]=\emptyset$ (these sets generate the Borel algebra of $\mc{G}$). The Liouville current has the property that 
\[
L_X(\gamma)=i(\mc{L},\delta_\gamma)
\]
for every $\gamma\in\Gamma$ (see \cite{Bo88}). Therefore, $i(\mc{L},\bullet)$ extends continuously the length function $L_X(\bullet)$ to the space of geodesic currents.
\end{dfn}

\subsection{The ${\rm PSL}_2(\mb{R})$ model of $\mb{H}^{2,1}$}
The second central object that we discuss is the anti de Sitter 3-space $\mb{H}^{2,1}$. We will mostly work in its linear and projective models which we now describe. For more details on the material we present here, we refer the reader to \cite{BS20}. 

The group ${\rm SL}_2(\mb{R})$ sits inside the vector space of $2\times 2$ matrices with real entries $M_2(\mb{R})$ as the hyperboloid of vectors of norm $-1$ for the quadratic form $\langle\bullet,\bullet\rangle$ of signature $(2,2)$ given by 
\[
4\langle X,Y\rangle:={\rm det}(X)+{\rm det}(Y)-{\rm det}(X+Y)=-{\rm tr}(XY^\star).
\]
where ${\tiny \left[
\begin{array}{c c}
a &b\\
c &d\\
\end{array}
\right]^\star
:=
\left[
\begin{array}{c c}
d &-b\\
-c &a\\
\end{array}
\right]}$. 

Note that for every $X\in{\rm SL}_2(\mb{R})$, the restriction of the quadratic form to $T_X{\rm SL}_2(\mb{R})=X^\perp$ has signature $(2,1)$ and, hence, induces a $(2,1)$-pseudo-Riemannian metric on ${\rm SL}_2(\mb{R})$ (experts will have recognized the Killing form of ${\rm SL}_2(\mb{R})$). The group ${\rm SL}_2(\mb{R})\times {\rm SL}_2(\mb{R})$ acts on $M_2(\mb{R})$ by left and right multiplications as $(A,B)\cdot X:=AXB^{-1}$ and the action is isometric with respect to $\langle\bullet,\bullet\rangle$. Passing to the projectivization, ${\rm PSL}_2(\mb{R})\subset\mb{P}(M_2(\mb{R}))$ we obtain the projective model of anti de Sitter 3-space $\mb{H}^{2,1}$.

\subsubsection{Boundary at infinity}
In this model, the boundary at infinity $\partial\mb{H}^{2,1}$ of $\mb{H}^{2,1}$ identifies with the topological boundary of ${\rm PSL}_2(\mb{R})$ in $\mb{P}(M_2(\mb{R}))$
\[
\partial{\rm PSL}_2(\mb{R})=\{[X]\in\mb{P}(M_2(\mb{R}))\left|\;{\rm det}(X)=0\right.\}.
\]

Observe that $\partial{\rm PSL}_2(\mb{R})$ consists of rank one matrices and can be naturally ${\rm PSL}_2(\mb{R})\times {\rm PSL}_2(\mb{R})$-equivariantly identified with $\mb{RP}^1\times\mb{RP}^1$ via the map
\[
\begin{array}{l c}
&\partial{\rm PSL}_2(\mb{R})\to\mb{RP}^1\times\mb{RP}^1\\
&[X]\to([{\rm Im}(X)],[{\rm Ker}(X)]).\\
\end{array}
\]

\subsubsection{Subspaces}
Totally geodesic subspaces in anti de Sitter 3-space $\mb{H}^{2,1}$ are of the form $\mb{P}(V)\cap{\rm PSL}_2(\mb{R})$ where $V\subset M_2(\mb{R})$ is a linear subspace intersecting ${\rm SL}_2(\mb{R})$. In particular we have
\begin{itemize}
\item{{\em timelike geodesics} isometric to $\mb{R}/\pi\mb{Z}\Leftrightarrow V$ 2-plane of signature $(0,2)$.}
\item{{\em spacelike geodesics} isometric to $\mb{R}\Leftrightarrow V$ 2-plane of signature $(1,1)$.}
\item{{\em spacelike planes} isometric to $\mb{H}^2\Leftrightarrow V$ 3-plane of signature $(2,1)$.}
\end{itemize} 

Two distinct points $x,y\in\mb{H}^{2,1}$ are joined by: 
\begin{itemize}
\item{A spacelike geodesic if and only if $|\langle x,y\rangle|>1$.}
\item{A timelike geodesic if and only if $|\langle x,y\rangle|<1$.}
\end{itemize}

The geodesic $\gamma(t)$ starting at $x\in\mb{H}^{2,1}$ with velocity $v\in T_x\mb{H}^{2,1}=x^\perp$ is parametrized by
\[
\gamma(t)=\left\{
\begin{array}{l l}
\cosh(t)x+\sinh(t)v &\text{\rm if $\langle v,v\rangle=1$},\\
%x+tv &\text{\rm if $|v|^2=0$},\\
\cos(t)x+\sin(t)v &\text{\rm if $\langle v,v\rangle=-1$}.\\
\end{array}
\right.
\]

\subsubsection{Acausal sets and pseudo-metrics}
The last concept that we need is the one of acausality:

\begin{dfn}[Acausal Set]
A subset $S\subset\mb{H}^{2,1}\cup\partial\mb{H}^{2,1}$ is {\em acausal} if for every $x,y\in S$ the geodesic $[x,y]$ is spacelike.
\end{dfn}

\begin{dfn}[Pseudo Metric]
On acausal subsets $S\subset\mb{H}^{2,1}$ we have a {\em pseudo-metric} $d_{\mb{H}^{2,1}}(\bullet,\bullet)$ defined as follows
\[
\cosh(d_{\mb{H}^{2,1}}(x,y))=|\langle x,y\rangle|.
\]
Notice that $d_{\mb{H}^{2,1}}$ does not satisfy the triangle inequality in general.
\end{dfn}

%Acausal subsets $S\subset\partial\mb{H}^{2,1}=\mb{RP}^1\times\mb{RP}^1$ are always graphs of strictly monotone maps $h:T\subset\mb{RP}^1\to\mb{RP}^1$ where $T$ is the projection of $S$ onto the first factor.

\section{Mess representations and pleated surfaces}
\label{sec:sec7}

The goal of the section is to describe Mess representations and the geometry of their pleated surfaces. In particular, at the end of the section, we discuss the structure of the boundary of the convex core associated with a Mess representation.

\subsection{Mess representations}
First of all we introduce the following class:

\begin{dfn}[Mess Representation]
Let $X,Y\in\T$ be hyperbolic structures. The {\em Mess representation} with parameters $X,Y$ is 
\[
\rho_{X,Y}:=(\rho_X,\rho_Y):\Gamma\to{\rm PSL}_2(\mb{R})\times{\rm PSL}_2(\mb{R})
\]
where $\rho_X,\rho_Y$ are the holonomy representations of $X,Y$.
\end{dfn}

\subsubsection{Boundary maps}
Every Mess representation $\rho_{X,Y}$ comes with a natural equivariant boundary map
\[
\xi:\partial\Gamma\to\partial\mb{H}^{2,1}
\]

It can be described explicitly as follows: Recall that $\partial{\rm PSL}_2(\mb{R})$ is naturally identified with $\mb{RP}^1\times\mb{RP}^1$. Let $h_X,h_Y:\partial\Gamma\to\mb{RP}^1$ be the unique $\rho_X,\rho_Y$-equivariant homeomorphism. The boundary map $\xi:\partial\Gamma\to\mb{RP}^1\times\mb{RP}^1$ is just $\xi=(h_X,h_Y)$. 

Its image $\xi(\partial\Gamma)=\Lambda_{X,Y}$ is the graph of the unique $(\rho_X-\rho_Y)$-equivariant homeomorphism $h_{X,Y}:\mb{RP}^1\to\mb{RP}^1$. 

Checking that $\Lambda_{X,Y}$ has the property that for every $a,b,c\in\mb{RP}^1$, the 3-space ${\rm Span}\{(a,h_{X,Y}(a)),(b,h_{X,Y}(b)),(c,h_{X,Y}(c))\}$ has signature $(2,1)$ is not difficult: Let us assume without loss of generality that $a<b<c$. As $h_{X,Y}$ is an orientation preserving homeomorphism, we have $h_{X,Y}(a)<h_{X,Y}(b)<h_{X,Y}(c)$. Hence, up to the action of ${\rm PSL}_2(\mb{R})\times{\rm PSL}_2(\mb{R})$, we can assume that $a,b,c=h_{X,Y}(a),h_{X,Y}(b),h_{X,Y}(c)=0,1,\infty$. Tracing back the identification with $\partial{\rm PSL}_2(\mb{R})$ we see that 
\[
{\tiny (0,0)=\left[
\begin{array}{c c}
0 &0\\
1 &0\\
\end{array}
\right],\; (1,1)=\left[
\begin{array}{c c}
1 &-1\\
1 &-1\\
\end{array}
\right],\; (\infty,\infty)=\left[
\begin{array}{c c}
0 &1\\
0 &0\\
\end{array}
\right]
}.
\]
The conclusion follows by an elementary computation.

\subsubsection{Domain of discontinuity}
From the boundary curve $\Lambda_{X,Y}\subset\partial\mb{H}^{2,1}$ one constructs a standard open domain: 
\[
\Omega_{X,Y}:=\{y\in\mb{H}^{2,1}\left|\;\text{$[x,y]$ spacelike $\forall x\in\Lambda_{X,Y}$}\right.\}
\]
It can also be described as a connected component of
\[
\mb{H}^{2,1}-\bigcup_{x\in\Lambda_{X,Y}}{\{\langle x,\bullet\rangle=0\}}
\]
which is a properly convex subset of $\mb{P}(M_2(\mb{R}))$ whose closure contains $\Lambda_{X,Y}$. In particular, it contains a natural closed $\rho_{X,Y}(\Gamma)$-invariant convex subset, namely the convex hull $\mc{CH}_{X,Y}$ of the limit set $\Lambda_{X,Y}$.

As $\Omega_{X,Y}$ does not contain closed timelike geodesics, it has a well defined timelike distance:

\begin{dfn}[Timelike Distance]
The {\em timelike distance} $\delta_{\mb{H}^{2,1}}(\bullet,\bullet)$ on $\Omega_{X,Y}$ is defined by
\[
\cos(\delta_{\mb{H}^{2,1}}(x,y)):=
\left\{
\begin{array}{l l}
|\langle x,y\rangle| &\text{ if $[x,y]$ is timelike}\\
1 &\text{ otherwise}.\\
\end{array}
\right.
\]
\end{dfn}

The group $\rho_{X,Y}(\Gamma)$ acts freely and properly discontinuosly on $\Omega_{X,Y}$ (see \cite{M07}). The quotient $M_{X,Y}:=\Omega_{X,Y}/\rho_{X,Y}(\Gamma)$ is the {\em Mess manifold} associated with $X,Y\in\T$. 

Let us mention the fact that $M_{X,Y}$ is a so-called {\em globally hyperbolic maximal Cauchy compact} anti de Sitter 3-manifold (GHMC). In particular, this means that $M_{X,Y}$ contains a closed spacelike surface $S$ homeomorphic to $\Sigma$ which intersects every inextensible timelike geodesic exactly once. From this property it is not difficult to deduce that $M_{X,Y}$ is diffeomorphic to $\Sigma\times\mb{R}$. Mess proves in \cite{M07} that, in fact, all GHMC manifolds $M$ where the Cauchy surface is homeomorphic to $\Sigma$ have the form $M=M_{X,Y}$ for some $X,Y\in\T$. 

\subsection{Laminations and pleated surfaces}
Mess representations are examples of {\em maximal representations} in ${\rm PSL}_2(\mb{R})\times{\rm PSL}_2(\mb{R})={\rm PSO}_0(2,2)$ as introduced in \cite{BIW10} (in fact, by a celebrated result of Goldman \cite{G80}, every maximal representation in ${\rm PSL}_2(\mb{R})\times{\rm PSL}_2(\mb{R})$ is a Mess representation). 

As a consequence, we can apply the results of \cite{MV22a} to our setting. In this section, we recall the pleated surface construction from \cite{MV22a} and describe some geometric properties of these objects.

\subsubsection{Pleated sets}

Let $\rho_{X,Y}$ be a Mess representation with boundary map $\xi:\partial\Gamma\to\Lambda_{X,Y}$.

\begin{dfn}[Geometric Realization]
Let $\lambda\in\mc{GL}$ be a lamination. The {\em geometric realization} of $\lambda$ for $\rho_{X,Y}$ is
\[
{\hat \lambda}:=\bigcup_{(a,b)\in\lambda}{[\xi(a),\xi(b)]}\subset\mc{CH}_{X,Y}
\]
where $(a,b)$ is the leaf of $\lambda$ with endpoints $a,b$ and $[\xi(a),\xi(b)]$ is the spacelike geodesic with endpoints $\xi(a),\xi(b)$.
\end{dfn}

\begin{dfn}[Pleated Set]
Let $\lambda\in\mc{GL}_m$ be a maximal lamination. The {\em pleated set} associated with $\lambda$ and $\rho_{X,Y}$ is
\[
{\hat S}_\lambda:={\hat \lambda}\cup\bigcup_{\Delta(a,b,c)\subset\mb{H}^2-\lambda}{\Delta(\xi(a),\xi(b),\xi(c))}\subset\mc{CH}_{X,Y}
\]
where $\Delta(a,b,c)$ is the plaque of $\lambda$ with vertices $a,b,c$ and $\Delta(\xi(a),\xi(b),\xi(c))$ is the ideal spacelike triangle with endpoints $\xi(a),\xi(b),\xi(c)$.
\end{dfn}

\begin{pro}[Proposition 3.7 of \cite{MV22a}] 
\label{pro:existence pleated sets}
The pleated set $\widehat{S}_\lambda\subset\mc{CH}_{X,Y}$ is a $\rho_{X,Y}(\Gamma)$-invariant topological Lipschitz acausal subsurface.
\end{pro}

Incidentally, combining with classical 3-dimensional topology, Proposition \ref{pro:existence pleated sets} has also the following topological corollary: 

\begin{cor}
\label{cor:isotopy}
Let $\rho_{X,Y}$ be a Mess representation with parameters $X,Y\in\T$. Identify the Mess manifold $M_{X,Y}:=\Omega_{X,Y}/\rho_{X,Y}(\Gamma)$ with $\Sigma\times\mb{R}$. Let $\alpha\subset \Sigma$ be an essential multicurve. Then, the geodesic realization of $\alpha$ in $M_{X,Y}$ is isotopic to $\alpha\subset\Sigma\times\{0\}$.
\end{cor}

\begin{proof}
Let $\lambda_\alpha$ be a maximal lamination obtained from $\alpha$ by adding finitely many geodesics spiraling around the curves in $\alpha$. By Proposition \ref{pro:existence pleated sets} there exists an embedded $\pi_1$-injective (Lipschitz) surface $S_{\alpha}={\hat S}_{\lambda_\alpha}/\rho_{X,Y}(\Gamma)\subset M_{X,Y}$ containing the geodesic realization of the curves in $\alpha$. By Proposition 3.1 and Corollary 3.2 of \cite{W68}, such surface, being embedded and $\pi_1$-injective, is isotopic to $\Sigma\times\{0\}$.
\end{proof}

\subsubsection{Bending locus}
The pleated set ${\hat S}_\lambda$ is not necessarily bent along all the lines in ${\hat \lambda}$.

\begin{dfn}[Bending Locus]
Let $\rho_{X,Y}$ be a Mess representation. Consider $\lambda$ a maximal lamination with geometric realization ${\hat \lambda}$, and denote by $\widehat{S}_\lambda$ the corresponding pleated set. A point $x\in\ell\subset{\hat \lambda}$ is in the \emph{bending locus of $\widehat{S}_\lambda$} if there is no (necessarily spacelike) geodesic segment $k$ entirely contained in $\widehat{S}_\lambda$ and such that ${\rm int}(k)\cap\ell=x$. 
\end{dfn}

We have:

\begin{pro}[Proposition 3.11 of \cite{MV22a}]
\label{pro:bending locus}
The bending locus is a sublamination of ${\hat \lambda}$, and its complement in $\widehat{S}_\lambda$ is a union of 2-dimensional totally geodesic spacelike regions.
\end{pro}

\subsubsection{1-Lipschitz developing map}
Unfolding pleated sets along the bending locus naturally maps them to $\mb{H}^2$. We formalize this heuristic picture as follows:

\begin{dfn}[Developing Map]
\label{def:developing map}
Let $\rho_{X,Y}$ be a Mess representation. Let ${\hat S}_\lambda\subset\mc{CH}_{X,Y}$ be the pleated set associated with the maximal lamination $\lambda$. A {\em 1-Lipschitz developing map} is a homeomorphism $f:{\hat S}_\lambda\to\mb{H}^2$ with the following properties:
\begin{enumerate}
\item{It is totally geodesic on every leaf of ${\hat \lambda}$ and every plaque.}
\item{It is 1-Lipschitz with respect to the intrinsic pseudo-metric on ${\hat S}_\lambda$ and the hyperbolic metric on $\mb{H}^2$.}
\end{enumerate}
\end{dfn}

Developing maps have a couple of useful general properties which we now describe. First, they are totally geodesic outside the bending locus.

\begin{lem}[Lemma 6.2 of \cite{MV22a}]
\label{lem:totally geodesic outside bending}
Let $\rho_{X,Y}$ be a Mess representation, and let ${\hat S}_\lambda$ be the pleated set associated to a maximal lamination $\lambda$. Then every 1-Lipschitz developing map $f:{\hat S}_\lambda\to\mb{H}^2$ is totally geodesic on the complement of the bending locus of ${\hat S}_\lambda$.
\end{lem}

Secondly, developing maps are contracting with respect to the natural path metric structure on pleated sets.

\begin{dfn}[Regular Path]
\label{def:length space}    
A {\em (weakly) regular path} is a map $\gamma:I=[a,b]\to\mb{H}^{2,1}$ such that: 
\begin{itemize}
\item{The path is Lipschitz}.
\item{The tangent vector ${\dot \gamma}(t)$ is spacelike (or lightlike) for almost every $t\in I$ (at which $\dot \gamma$ is defined).}
\end{itemize}
The length of a weakly regular path is 
\[
L(\gamma):=\int_I{\sqrt{\langle{\dot \gamma}(t),{\dot \gamma}(t)\rangle} dt}.
\]
The Lipschitz property implies that the length $L(\gamma)$ is always finite.
\end{dfn}

\begin{lem}[Claim 2 of Lemma 6.4 in \cite{MV22a}]
\label{lem:path length}
Let ${\hat S}\subset\mb{H}^{2,1}$ be an acausal subset. Let $\gamma:I=[a,b]\to{\hat S}$ be a weakly regular path. Then
\[
L(\gamma)=\lim_{\ep\to 0}\int_I{\frac{d_{\mb{H}^{2,1}}(\gamma(t),\gamma(t+\ep))}{\ep}{ \rm dt}}.
\]
\end{lem}

\begin{lem}[Lemma 6.4 of \cite{MV22a}]
\label{lem:rectifiable to rectifiable}
Let $\rho_{X,Y}$ be a Mess representation, and let ${\hat S}_\lambda$ be the pleated set associated to a maximal lamination $\lambda$. Then every 1-Lipschitz developing map $f:{\hat S}_\lambda\to\mb{H}^2$ sends weakly regular paths $\gamma:I\to {\hat S}_\lambda$ to Lipschitz (hence rectifiable) paths $f\gamma:I\to \mb{H}^2$ of smaller length $L(\gamma)\ge L(f\gamma)$.
\end{lem}

\subsubsection{Pleated surfaces}
The following result makes sure that every pleated set ${\hat S}_\lambda$ admits a natural 1-Lipschitz developing map:

\begin{pro}[Proposition 6.6 in \cite{MV22a}]
\label{pro:hyperbolic structure}
Let $\rho_{X,Y}$ be a Mess representation. For every maximal lamination $\lambda\in\mc{GL}_\lambda$ there is:
\begin{itemize}
\item{An intrinsic hyperbolic structure $Z_\lambda\in\T$.}
\item{A $(\rho_{X,Y}-\rho_\lambda)$-equivariant 1-Lipschitz developing map $f:{\hat S}_\lambda\to\mb{H}^2$ where $\rho_\lambda$ is the holonomy of $Z_\lambda$.}
\end{itemize} 
\end{pro}

We can finally define pleated surfaces:

\begin{dfn}[Pleated Surface]
\label{def:pleated surface}
Let $\rho_{X,Y}$ be a Mess representation. The {\em pleated surface} associated with the maximal lamination $\lambda\in\mc{GL}$ consists of the following data:
\begin{enumerate}
\item{The pleated set ${\hat S}_\lambda$.}
\item{The intrinsic hyperbolic holonomy $\rho_\lambda:\Gamma\to{\rm PSL}_2(\mb{R})$ of $Z_\lambda$.}
\item{A $(\rho_{X,Y}-\rho_\lambda)$-equivariant 1-Lipschitz developing map $f:{\hat S}_\lambda\to\mb{H}^2$.}
\end{enumerate}
\end{dfn}

Let us conclude this discussion by observing that pleated surfaces for a fixed Mess representation $\rho_{X,Y}$ have some useful compactness properties:

\begin{lem}
\label{lem:pleated cpt}
Let $\rho_{X,Y}$ be the Mess representation with parameters $X,Y\in\T$. Then the space of intrinsic hyperbolic structures on the pleated sets
\[
\{Z_\lambda\}_{\lambda\in\mc{GL}_m}
\]
is pre-compact in $\T$.
\end{lem}

\begin{proof}
Recall that $\rho_{X,Y}(\Gamma)$ acts cocompactly on $\mc{CH}_{X,Y}$. Let $F\subset\mc{CH}_{X,Y}$ be a compact fundamental domain. For every maximal lamination $\lambda\in\mc{GL}_m$ with associated pleated set ${\hat S}_\lambda\subset\mc{CH}_{X,Y}$ choose a basepoint $x_\lambda\in{\hat S}_\lambda\cap F$. Let $f_\lambda:{\hat S}_\lambda\to\mb{H}^2$ be a $(\rho_{X,Y}-\rho_\lambda)$-equivariant 1-Lipschitz developing map normalized so that $f_\lambda(x_\lambda)=o\in\mb{H}^2$, a fixed basepoint. The equivariance and the 1-Lipschitz property tell us that 
\[
d_{\mb{H}^2}(o,\rho_\lambda(\gamma)o)\le d_{\mb{H}^{2,1}}(x_\lambda,\rho_{X,Y}(\gamma)x_\lambda)
\]
for every $\gamma\in\Gamma$. Notice that the right hand side is bounded from above by a uniform constant $K_\gamma$ independent of $\lambda$ as $x_\lambda\in F$ is contained in a compact set and 
\[
\cosh(d_{\mb{H}^{2,1}}(x_\lambda,\rho_{X,Y}(\gamma)x_\lambda))=|\langle x_\lambda,\rho_{X,Y}(\gamma)x_\lambda\rangle|.
\]
Therefore the set of representations $\{\rho_\lambda\}_{\lambda\in\mc{GL}_m}\subset\T\subset{\rm Hom}(\Gamma,{\rm PSL}_2(\mb{R}))$ is pre-compact.
\end{proof}

\subsubsection{Convex core}
An example of pleated surface is given by the two connected components of the boundary of the convex core $\partial\mc{CH}_{X,Y}=\partial^+\mc{CH}_{X,Y}\cup\partial^-\mc{CH}_{X,Y}$. Each of them has the structure of a pleated set with bending loci $\lambda^+$ and $\lambda^-$ and intrinsic hyperbolic structures $Z_{\lambda^+},Z_{\lambda^-}\in\T$. As we mentioned in the introduction, measuring the total turning angles along paths $\alpha:I\to\partial^\pm\mc{CH}_{X,Y}$ equips the geodesic laminations $\lambda^\pm$ with a transverse measure and, hence identifies a pair of points $\lambda^\pm\in\mc{ML}$. Mess proves that we have the following relations
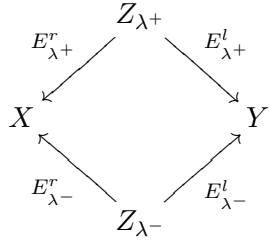
\[
\xymatrix{
&Z_{\lambda^+}\ar[dl]_{E_{\lambda^+}^r}\ar[dr]^{E_{\lambda^+}^l} &\\
X & &Y\\
&Z_{\lambda^-}\ar[ul]^{E_{\lambda^-}^r}\ar[ur]_{E_{\lambda^-}^l} &
}
\]
where $E_{\lambda^+}^l,E_{\lambda^-}^l,E_{\lambda^+}^r,E_{\lambda^-}^r$ are the {\rm left} and {\rm right earthquakes} induced by the measured laminations $\lambda^+,\lambda^-$. Heuristically speaking, an earthquake is the generalization to laminations of a {\em twist deformation} along a simple closed geodesic. Given a closed geodesic $\gamma$ on a hyperbolic surface $X$ and a real parameter $\theta>0$ we do the following operation: We lift $\gamma$ to a $\rho_X(\Gamma)$-invariant family of pairwise disjoint geodesics $\lambda\subset\mb{H}^2$. We cut $\mb{H}^2$ along $\lambda$. We reglue all the ideal polygons $P\subset\mb{H}^2-\lambda$ by composing all the initial identifications $\ell\subset\partial P\to\ell'\subset\partial P'$ (left-to-right) with the isometry of $\ell'$ given by $t\to t+\theta$ (the identification $\ell'=\mb{R}$ is determined by the boundary orientation). The result is still isometric to $\mb{H}^2$ but the action of $\Gamma$ on it is the holonomy of a different hyperbolic structure, which, depending on the choices of orientations, is $E^l_{\theta\gamma}(X)$ or $E^r_{\theta\gamma}(X)$.

We will describe more carefully the various elements that enter this picture in the next section where we will prove a generalization of the result of Mess.

\subsubsection{Initial and terminal singularities}
We end this section by describing the {\em initial} and {\em terminal singularities} of $\Omega_{X,Y}$ which are subsets of $\partial\Omega_{X,Y}$ dual to the boundary components of the convex core. Duality is understood in the sense of the duality induced by the quadratic form $\langle\bullet,\bullet\rangle_{(2,2)}$ on $\mb{P}(M_2(\mb{R}))$. Explicitly, we have 
\[
\mb{P}(L)\leftrightarrow\mb{P}(L^\perp)
\]
where $L^\perp\subset M_2(\mb{R})$ is the linear subspace orthogonal to $L$ with respect to the quadratic form. 

Define the following:

\begin{dfn}[Initial and Terminal Singularities]
The sets $\mc{S}^\pm$ of dual points of supporting planes of $\partial^\pm\mc{CH}_{X,Y}$ are the {\em initial} and {\em terminal singularities}.
\end{dfn} 

Let us start with the following observation: 

\begin{lem}
\label{lem:supporting plane}
Let $H=P(V)\cap\mb{H}^{2,1}$ be a supporting plane of $\partial^\pm\mc{CH}_{X,Y}$. Then:
\begin{itemize}
\item{$H$ is spacelike and defines a dual point $P(V^\perp)\in\mb{H}^{2,1}$. Let $w\in V^\perp$ be a unit timelike vector pointing outside $\mc{CH}_{X,Y}$.}
\item{For every $x\in H\cap\mc{CH}_{X,Y}$, the timelike geodesic $\gamma(t)=\cos(t)x-\sin(t)w$ with $t\in[0,\pi/2)$ is contained in $\Omega_{X,Y}$ while $w=\gamma(\pi/2)\in\partial\Omega_{X,Y}$.} 
\end{itemize}

Any two distinct supporting planes $H_1,H_2$ of $\partial^+\mc{CH}_{X,Y}$ intersect in a spacelike geodesic $H_1\cap H_2$. If $w_1,w_2$ are the dual points of $H_1,H_2$, then $[w_1,w_2]$ is spacelike.
\end{lem}

\begin{proof}
The first point: Recall that $\partial\mb{H}^{2,1}=\mb{RP}^1\times\mb{RP}^1$ and that $\Lambda_{X,Y}$ is the graph of an orientation preserving homeomorphism $h_{X,Y}:\mb{RP}^1\to\mb{RP}^1$. If $H$ is a supporting hyperplane for $\mc{CH}_{X,Y}$ then $\partial H$ does not intersect $\Lambda_{X,Y}$ transversely. The fact that $H$ must be spacelike follows from the following observations: The boundary of a lightlike plane has the form $\{t\}\times\mb{RP}^1$ or $\mb{RP}^1\times\{t\}$. The boundary of a totally geodesic plane in $\mathbb{H}^{2,1}$ with signature $(1,1)$ is the graph of an orientation reversing linear transformation $\mb{RP}^1\to\mb{RP}^1$. In both cases the boundary intersects $\Lambda_{X,Y}$ transversely.

The second point: Recall that $\Omega_{X,Y}$ is the set of points that can be connected to every point in $\Lambda_{X,Y}$ by a spacelike geodesic. A point $x\in\mb{H}^{2,1}$ and a point $p\in\partial\mb{H}^{2,1}$ are connected by a spacelike geodesic if and only if $\langle x,p\rangle\neq 0$. Let us show that $\gamma(t)\in\Omega_{X,Y}$ for every $t\in[0,\pi/2)$. In order to do so, lift $\Lambda_{X,Y}$ continuously to $M_2(\mb{R})$. As $x\in\Omega_{X,Y}$, we have $\langle x,p\rangle\neq 0$ for every $p\in\Lambda_{X,Y}$ and, by continuity, we can assume that it is negative for every $p\in\Lambda_{X,Y}$. As $H$ is a supporting hyperplane and $w$ is timelike, orthogonal to $H$, and pointing outside $\mc{CH}_{X,Y}$, we have $\langle p,w\rangle\ge 0$ for every $p\in\Lambda_{X,Y}$. Therefore $\langle\gamma(t),p\rangle=\cos(t)\langle x,p\rangle-\sin(t)\langle w,p\rangle<0$ for every $p\in\Lambda_{X,Y}$ and $t<\pi/2$. In order to conclude, it is enough to observe that $w=\gamma(\pi/2)\not\in\Omega_{X,Y}$ as $\langle w,p\rangle=0$ for every $p\in\partial H\cap\Lambda_{X,Y}\neq\emptyset$.

For the last part notice that $H_1\cap H_2$ is either empty or a spacelike geodesic. Suppose that $H_1\cap H_2=\emptyset$. Then $\mb{H}^{2,1}-(H_1\cup H_2)$ consists of two connected components one of them containing $\mc{CH}_{X,Y}$. As $H_1,H_2$ lie on opposite sides of $\mc{CH}_{X,Y}$ in such component, they cannot be supporting hyperplanes for the same boundary component of $\partial\mc{CH}_{X,Y}$. This is a contradiction.
\end{proof}

Notice that, by Lemma \ref{lem:supporting plane}, the initial and terminal singularities $\mc{S}^\pm$ are $\rho_{X,Y}(\Gamma)$-invariant, acausal, and contained in $\partial\Omega_{X,Y}$. Benedetti and Guadagnini \cite{BG} prove that they have the structure of a $\mb{R}$-tree and relate them to the bending laminations $\lambda^\pm$.

\begin{dfn}[$\mb{R}$-tree]
A {\em $\mb{R}$-tree} is a geodesic metric space $(\mc{S},d_{\mc{S}}(\bullet,\bullet))$ such that between two points $x,y\in\mc{S}$ there is a unique (up to reparametrization) injective path $\alpha:[0,1]\to\mc{S}$ with $\alpha(0)=x,\alpha(1)=y$.
\end{dfn}

Benedetti and Guadagnini \cite{BG} show the following:

\begin{pro}
\label{pro:singularities}
Let $\rho_{X,Y}$ be a Mess representation. Let $\mc{S}^\pm\subset\partial\Omega_{X,Y}$ be the initial and terminal singularities. Then: 
\begin{itemize}
\item{$\mc{S}^\pm$ is $\rho_{X,Y}(\Gamma)$-invariant, acausal, and path connected by regular paths. In particular, it has an intrinsic path metric 
\[
d_{\mc{S}^\pm}(x,y)=L(\alpha)
\]
where $\alpha:[0,1]\to\mc{S}^\pm$ is a regular path joining $x$ to $y$.}
\item{For every pair of points $w,w'\in\mc{S}^\pm$ there is a unique continuous injective path connecting them.}
\item{For every $\gamma\in\Gamma-\{1\}$, the minimal displacement 
\[
\min_{x\in\mc{S}^\pm}\{d_{\mc{S}^\pm}(x,\rho_{X,Y}(\gamma)x)\}
\]
coincides with $i(\gamma,\lambda^\pm)$ and is realized by some point $x\in\mc{S}^\pm$.}
\end{itemize} 
Here $\lambda^\pm\in\mc{ML}$ is the bending lamination of $\partial^\pm\mc{CH}_{X,Y}$ and $i(\bullet,\bullet)$ is the geometric intersection form.
\end{pro}

For a proof we refer to \cite{BG} and \cite{BB09}.

\section{A generalization of a result of Mess}
\label{sec:sec4}

The goal of the section is to define the shear-bend cocycles of pleated surfaces and prove Theorem \ref{thm:pleated surfaces ads}.

We begin by recalling the Thurston-Bonahon shear parametrization of Teichmüller space (as discussed by Bonahon in \cite{Bo96}) which we will generalize to the space of Mess representations in Theorem \ref{thm:shear-bend ads} at the end of the section.

\subsection{Shear coordinates}

We refer to Bonahon \cite{Bo96} for more details on the material presented in this section.

\subsubsection{Transverse cocycles}
Shear-bend cocycles are a special case of transverse cocycles for $\lambda$.

\begin{dfn}[Transverse Cocycle]
Let $\mb{A}$ be a commutative ring. Let $\lambda\subset\mb{H}^2$ be a maximal lamination. An {\em $\mb{A}$-transverse cocycle} for $\lambda$ is a function $\sigma(\bullet,\bullet)$ of pairs of plaques satisfying the following properties:
\begin{itemize}
\item{Invariance: $\sigma(\gamma P,\gamma Q)=\sigma(P,Q)$ for every $\gamma\in\Gamma$ and plaques $P,Q$.}
\item{Symmetry: $\sigma(P,Q)=\sigma(Q,P)$ for every plaques $P,Q$.}
\item{Additivity: $\sigma(P,R)=\sigma(P,Q)+\sigma(Q,R)$ for every plaques $P,Q,R$ such that $R$ separates $P$ from $Q$.}
\end{itemize}

The space of $\mb{A}$-transverse cocycles is denoted by $\mc{H}(\lambda;\mb{A})$. It has a natural structure of $\mb{A}$-module, which is isomorphic to $\mb{A}^{-3\chi(\Sigma)}$ whenever $\mathbb{A}$ has no $2$-torsion (see Bonahon \cite{Bo96}).
\end{dfn}

\subsubsection{Measured laminations}
Every measured lamination $\mu\in\mc{ML}_\lambda$ determines a natural transverse cocycle which, with a little abuse of notation, we will still denote by $\mu\in\mc{H}(\lambda;\mb{R})$. It is defined as follows: Let $P,P'$ be plaques of $\lambda$. Let $\ell \subset \partial P, \ell' \subset \partial P'$ be the (oriented) edges that separate $P,P'$. Then  
\[
\mu(P,P'):=\mu([\ell,\ell']),
\]
the measure, determined by $\mu$, of the box $[\ell,\ell']\subset\mc{G}$ consisting of those geodesics separating $\ell$ and $\ell'$.
 
\subsubsection{Hyperbolic structures}
Every hyperbolic structure $X$ on $\Sigma$ also determines a transverse cocycle $\sigma_\lambda^X\in\mc{H}(\lambda;\mb{R})$, the so-called {\em shear cocycle} of $X$. It is defined as follows: Let $P,P'$ be plaques of $\lambda$. Let $\ell\subset \partial P,\ell'\subset\partial P'$ be the (oriented) edges that separate $P,P'$. Denote by $x\in\ell,x'\in\ell'$ the orthogonal projections of the opposite vertices in $P,P'$. 

Consider the partial foliation $\lambda_{PP'}$ of the region $[\ell,\ell']$ bounded by $\ell,\ell'$ given by all the leaves that separate $P$ from $P'$ and note that $[\ell,\ell']-\lambda_{PP'}$ is a union of wedges, that is regions bounded by a pair of leaves of $\lambda_{PP'}$ that are asymptotic in one or the other direction. Each of the wedges can be foliated by adding all the geodesics separating the boundary leaves and to their common endpoint at infinity. Thus, we get a natural geodesic foliation of $[\ell,\ell']$. The line field on $[\ell,\ell']$ which is orthogonal to this foliation is integrable and following its leaves provides a natural isometric identification $\pi:\ell\to\ell'$. Define
\[
\sigma_\lambda^X(P,P'):=d_{\ell'}(\pi(x),x')
\]
where $d_{\ell'}$ is the signed distance along $\ell'$. 

A straightforward computation in $\mb{H}^2$ shows the following:

\begin{lem}
\label{lem:elementary case}
Let $\beta^\mb{R}$ be the cross ratio on $\mb{RP}^1$. We have
\begin{itemize}
\item{If $P,P'$ are adjacent triangles and $\ell=\ell'$, then
\[
\sigma_\lambda^X(P,P')= \log ( - \beta^\mb{R}(\ell^+,\ell^-,v,v'))
\]
where $v, v'$ are the ideal vertices of $P, P'$ opposite to $\ell=\ell'$, respectively, and $\ell$ is oriented so that $P$ and $P'$ lie on its left and on its right, respectively.}
\item{If $P,P'$ lie on opposite sides of a leaf $\ell\subset\lambda$ and each shares an ideal vertex with $\ell$, then  
\[
\sigma_\lambda^X(P,P')= \log( \beta^\mb{R}(\ell^+,u,v,\ell^-) \, \beta^\mb{R}(\ell^-,\ell^+,u',u) \, \beta^\mb{R}(\ell^-,u',v',\ell^+) )
\]
where: $u, u'$ are the ideal vertices of the sides $e\subset P,e'\subset P'$ that are not endpoints of $\ell$ and that separate the plaques $P,P'$; $v, v'$ are the vertices of $P, P'$ opposite to $e, e'$; $\ell$ is oriented so that $P$ and $P'$ lie on its left and on its right, respectively.}
\end{itemize}
\end{lem}
(See e.g. \cite{MV22a}*{Lemma 4.11} for a proof of the second assertion.) Bonahon proves the following:

\begin{thm}[Theorems A and B of \cite{Bo96}]
\label{thm:thurston bonahon}
Let $\lambda$ be a maximal lamination. For every $X\in\T$ the function $\sigma_\lambda^X(\bullet,\bullet)$ is a transverse cocycle. The map
\[
\begin{array}{c}
\Phi:\T\to\mc{H}(\lambda;\mb{R})\\
X\to\sigma_\lambda^X\\
\end{array}
\]
is a real analytic diffeomorphism. The image $\Phi(\T)$ is the open convex cone 
\[
\Phi(\T)=\{\sigma\in\mc{H}(\lambda,\mb{R})\left|\;\omega_{{\rm Th}}(\sigma,\bullet)>0\text{ \rm on }\mc{ML}_\lambda\right.\}
\]
where $\omega_{{\rm Th}}(\bullet,\bullet)$ is the {\rm Thurston's symplectic form} on $\mc{H}(\lambda;\mb{R})$.
\end{thm}

The resulting set of coordinates for Teichmüller space are called {\em shear coordinates} relative to $\lambda$. 

The Thurston's symplectic form $\omega_{{\rm Th}}(\bullet,\bullet)$ is a natural symplectic form on the vector space $\mc{H}(\lambda;\mb{R})$. For our purposes we don't need a precise definition of this object (we refer to Bonahon \cite{Bo96} for details), as we will only use the following property:

\begin{thm}[Theorem E of \cite{Bo96}]
\label{thm:symplectic length}
Let $\lambda$ be a maximal lamination. Let $\omega_{{\rm Th}}(\bullet,\bullet)$ is the Thurston's symplectic form on $\mc{H}(\lambda;\mb{R})$. Then, for every $\mu\in\mc{ML}_\lambda$ and $X\in\T$ we have
\[
\omega_{{\rm Th}}(\sigma_\lambda^X,\mu)=L_X(\mu).
\]
\end{thm}

\subsubsection{Continuity of cocycles}
In order to talk about continuity properties of cocycles we need to compare $\mc{H}(\lambda';\mb{R})$ with $\mc{H}(\lambda;\mb{R})$ for $\lambda'$ close to $\lambda$. This can be done using the {\em weights system} $\mc{W}(\tau;\mb{R})$ of a {\em train track} $\tau$ carrying $\lambda$. For us it is not important the definition of these objects, but rather the following facts (see the proof of Lemma 13 in Bonahon \cite{Bo98} or Proposition 5.10 and Corollary 5.11 in \cite{MV22a}):
\begin{itemize}
\item{$\tau$ determines an open set $U_\tau\subset\mc{GL}_m$ containing $\lambda$.}
\item{$\mc{W}(\tau;\mb{R})$ is a real vector space and there is a canonical linear isomorphism $\mc{H}(\lambda';\mb{R})\to\mc{W}(\tau;\mb{R})$ for every $\lambda'\in U_\tau$.}
\item{For every $\lambda_1,\lambda_2\in U_\tau$ the following diagram commutes
\[
\xymatrix{
\T\ar[d]\ar[r] &\mc{H}(\lambda_2;\mb{R})\ar[d]\\
\mc{H}(\lambda_1;\mb{R})\ar[r] &\mc{W}(\tau;\mb{R}).
}
\]
}
\item{For every $X\in\T$ the map $U_\tau \ni \lambda \mapsto \sigma^X_\lambda \in \mc{W}(\tau;\mb{R})$ is continuous.}
\end{itemize}

\subsection{Para-complex numbers}

In order to define the shear-bend cocycle of pleated surfaces it is convenient to exploit the natural para-complex cross-ratio on the boundary of $\mb{H}^{2,1}$ (see Section 2 of Danciger \cite{D14}).

\begin{dfn}[Para-complex Numbers]
The ring of para-complex numbers is $\mathbb{B}:=\mb{R}[\tau]/(\tau^2-1)$. Similarly to the case of complex numbers, every element $z=x+\tau y$ has: 
\begin{itemize}
\item{A conjugate ${\bar z}:=x-\tau y$.}
\item{A pseudo-norm $\abs{z}^2:=z\bar{z}=x^2-y^2\in\mb{R}$.}
\end{itemize}

However $\mb{B}$ has also non trivial zero-divisors: An element $z\in\mathbb{B}$ is invertible if and only if $\abs{z}^2\neq 0$, in which case $z^{-1}=\bar{z}/\abs{z}^2$. We denote by $\mb{B}^*$ the set of invertible elements of $\mathbb{B}$.
\end{dfn} 

It is convenient to decompose $\mb{B}$ as $\mb{R}\times\mb{R}$: Consider 
\[
e_l:=\frac{1+\tau}{2},e_r:=\frac{1-\tau}{2}.
\]
The elements $e_l,e_r$ are idempotent $e_j^2=e_j$, orthogonal $e_le_r = 0$, and conjugate ${\bar e_l}=e_r$. This implies that the map $(\lambda,\mu)\in\mb{R}\times\mb{R}\to\lambda e_l+\mu e_r\in\mathbb{B}$ is a ring isomorphism. In these coordinates, the conjugate of an element is $\overline{\lambda e_l+\mu e_r}=\mu e_l+\lambda e_r$ and its norm is $\abs{\lambda e_l+\mu e_r}=\lambda \mu$. 

\subsubsection{Exponential and logarithm}
The para-complex exponential function $\exp:\mb{B}\to\mb{B}$ is given by $\exp(z):=\sum_{k=0}^\infty\frac{z^k}{k!}$. In terms of the classical exponential we have $e^{x+\tau y}=e^x (\cosh(y)+\tau\sinh(y))$. The para-complex exponential map is injective, but not surjective. Its image coincides with
\[
\mathbb{B}^+ :=\{x+\tau y\in\mathbb{B}\mid\text{$x > 0$ and $\abs{x+\tau y}^2>0$}\}.
\]
The inverse of the exponential is the para-complex logarithm $\log:\mb{B}^+\to\mb{B}$. 

In coordinates $\mb{B}=\mb{R}\times\mb{R}$, we have: $\mb{B}^+=\{(\lambda,\mu)\in\mb{R}\times\mb{R}\left|\;\lambda,\mu>0\right.\}$. The exponential is $\exp(\lambda e_l+\mu e_r)=\exp(\lambda)e_l+\exp(\mu) e_r$. The logarithm is $\log\left(\lambda e_l+\mu e_r\right)=\log(\lambda)e_l+\log(\mu)e_r$.

\subsubsection{Projective para-complex line}
The boundary $\partial\mb{H}^{2,1}=\mb{RP}^1\times\mb{RP}^1$ can be identified with the para-complex projective line $\mb{BP}^1=(\mb{B}^2-\{0\})/\mb{B}^*$ via 
\[
([u],[v])\in\mb{RP}^1\times\mb{RP}^1\to\left[\frac{1+\tau}{2}u+\frac{1-\tau}{2}v\right]\in\mb{BP}^1
\]
and ${\rm PSL}_2(\mb{R})\times{\rm PSL}_2(\mb{R})$ can be thought of as the para-complex projective linear transformations ${\rm PSL}_2(\mb{B})={\rm SL}_2(\mb{B})/\mb{B}^*$ via the isomorphism 
\[
([A],[B])\in{\rm PSL}_2(\mb{R})\times{\rm PSL}_2(\mb{R})\to\left[\frac{1+\tau}{2}A+\frac{1-\tau}{2}B\right]\in{\rm PSL}_2(\mb{B}).
\]

The para-complex projective line $\mb{BP}^1$ is equipped with a natural {\em para-complex cross-ratio}:

\begin{dfn}[Cross Ratio]
The {\em para-complex cross-ratio} is defined by
\[
\beta^\mb{B}(z_1,z_2,z_3,z_4)=\frac{z_1-z_3}{z_1-z_4}\cdot\frac{z_2-z_4}{z_2-z_3}\in\mb{B} ,
\]
for any $4$-tuple $(z_1,z_2,z_3,z_4) \in \mathbb{B}^4$ such that $z_1 - z_4, z_2 - z_3 \in \mathbb{B}^*$.
\end{dfn}

The following is an elementary computation:

\begin{lem} 
\label{lem:para birapporto}
For every $a,b,c,d\in\mathbb{BP}^1=\mb{RP}^1\times\mb{RP}^1$ we have
\[
\beta^{\mb{B}}(a,b,c,d)=\frac{1+\tau}{2}\beta^{\mb{R}}(a_l,b_l,c_l,d_l)+\frac{1-\tau}{2}\beta^{\mb{R}}(a_r,b_r,c_r,d_r) ,
\]
where $x_l$ and $x_r$ denote the first and second components of $x \in \{a,b,c,d\} \subset \mb{RP}^1\times\mb{RP}^1$, respectively.
\end{lem}

\subsection{Shear-bend cocycle}

We now recall the natural shear-bend cocycle and its geometric interpretation as given in Sections 4 and 5 of \cite{MV22a}.

Let $\rho_{X,Y}$ be a Mess representation with limit curve $\Lambda_{X,Y}$.

\subsubsection{Elementary shear}
Let us start with an elementary shear-bend. 

\begin{lem}
\label{lem:positivity}
Let $\Delta=(u,\ell^-,\ell^+),\Delta'=(u',\ell^+,\ell^-)\subset\mb{H}^{2,1}$ be ideal triangles sharing a common edge $\ell=[\ell^-,\ell^+]$ and with vertices on $\Lambda_{X,Y}$ ordered as $u<\ell^-<u'<\ell^+$. Then $- \beta^{\mb{B}}(\ell^+,\ell^-,u,u')\in\mb{B}^+$.
\end{lem}

\begin{proof}
Recall that $\Lambda_{X,Y}$ is the graph of the unique $(\rho_X-\rho_Y)$-equivariant homeomorphism $h_{X,Y}:\mb{RP}^1\to\mb{RP}^1$. For a point $p\in\mb{RP}^1\times\mb{RP}^1$ denote by $p_l,p_r$ the left and right components. Then we have $u_j<\ell^+_j<u'_j<\ell^-_j$ on $\mb{RP}^1$ for $j=l,r$. The conclusion follows from Lemma \ref{lem:para birapporto}.
\end{proof}

We define:
\[
\sigma^{\mb{B}}(\Delta,\Delta'):=\log(-\beta^{\mb{B}}(\ell^+,\ell^-,u,u'))\in\mb{B}.
\]

\subsubsection{Maximal laminations with countably many leaves}
We then consider the case of maximal laminations with countably many leaves.

These laminations always have the following structure: There is a canonical collection of simple sublaminations
\[
\lambda'=\lambda_1\sqcup\cdots\sqcup\lambda_n\subset\lambda
\]
where each $\lambda_j$ consists of the orbit of the axis of an element $\gamma_j\in\Gamma-\{1\}$ representing a simple closed curve. The complement $\lambda-\lambda'$ is made of isolated geodesics asymptotic to leaves of $\lambda'$.

Let $\lambda\subset\mb{H}^2$ be a maximal lamination with countably many leaves, and let $P,Q\subset\mb{H}^2-\lambda$ be a pair of distinct plaques. We denote by $\ell_1,\cdots,\ell_m$ the leaves of $\lambda'$ separating $P$ from $Q$, oriented so that $P$ and $Q$ lie on the left and on the right of each $\ell_j$. For any $j \in \{1, \dots, m\}$, select two plaques $R_j^P$ and $R_j^Q$ that lie on the left and on right of $\ell_j$, respectively, and have an ideal vertex equal to one of the endpoints of $\ell_j$. The elementary shear between them can be computed via the same formal expression from Lemma \ref{lem:elementary case}, by simply replacing the role of $\beta^\mathbb{R}$ with $\beta^\mathbb{B}$ (see also Lemma \ref{lem:para birapporto}). Note that between $R_{j-1}^Q$ and $R_j^P$ there are only finitely many consecutive adjacent plaques 
\[
R_{j-1}^Q=T_{j,0},\cdots,T_{j,k_j}=R_j^P.
\]
Define
\[
\sigma_\rho(P,Q):=\sum_{j=1}^m{\left(\sigma^{\mb{B}}(R_j^P,R_j^Q)+\sum_{i=0}^{k_j-1}\sigma^{\mb{B}}(T_{j,i},T_{i+1})\right)}.
\]

As observed in \cite{MV22a}*{Section 4.4}, a simple cross ratio identity shows that a different choice of plaques $R_j^P,R_j^Q$ asymptotic from the left and from the right to the leaves $\ell_j\in\lambda'$ separating $P$ from $Q$ gives the same value for $\sigma_\rho(P,Q)$. The fact that $\sigma_\rho(P,Q)$ is well-defined immediately implies that it is also satisfies the properties of a transverse cocycle. Therefore:

\begin{dfn}[Intrinsic Shear-Bend I]
Let $\rho_{X,Y}$ be a Mess representation. Let $\lambda$ be a maximal lamination with countably many leaves. The cocycle $\sigma_\rho(\bullet,\bullet)\in\mc{H}(\lambda;\mb{B})$ is the {\em intrinsic shear-bend cocycle} of the pleated set ${\hat S}_\lambda$.
\end{dfn} 

Furthermore we have:

\begin{pro}[Proposition 6.7 in \cite{MV22a}]
\label{pro:shear is shear}
Let $\rho_{X,Y}$ be a Mess representation. Let $\lambda$ be a maximal lamination with countably many leaves. Then $\Re \sigma_\rho = (\sigma_\rho+{\bar\sigma}_\rho)/2\in\mc{H}(\lambda;\mb{R})$ is the shear cocycle of the intrinsic hyperbolic structure $Z_\lambda\in\T$ of the pleated set ${\hat S}_\lambda$.
\end{pro}

\begin{proof}
	Proposition 6.7 in \cite{MV22a} characterizes the intrinsic hyperbolic structure $Z_\lambda\in\T$ in terms of its shear coordinates. More precisely, one sees that the shear cocycle of $Z_\lambda$ can be reconstructed from a natural cross ratio $\beta_\xi$ associated to the $\rho_{X,Y}$-equivariant limit map $\xi : \partial \Gamma \to \partial \mathbb{H}^{2,1}$, and defined purely in terms of the preudo-Riemannian structure of $\mathbb{R}^{2,2}$. An elementary computation shows that the cross ratios $\beta_\xi$ and $\beta^{\mathbb{B}}$ are related by the identity
	\[
	| \beta^{\mathbb{B}}(\xi(x_1),\xi(x_2),\xi(x_3),\xi(x_4)) |^2 = \beta_\xi(x_1,x_2,x_3,x_4)^2 ,
	\]
	as we identify $\partial \mathbb{H}^{2,1}$ with $\mathbb{B}\mathrm{P}^1$. (Here $| \cdot |^2$ denotes the natural pseudo-norm on $\mathbb{B}$.) It then follows from the properties of the para-complex logarithm that the shear $\sigma_\xi$, defined through the cross ratio $\beta_\xi$, and the para-complex shear $\sigma^\mathbb{B}$ satisfy
	\[
	\Re(\sigma^{\mathbb{B}}(P,Q)) = \sigma_\xi(P,Q) ,
	\]
	for any pair of distinct plaques $P, Q$ of $\lambda$, from which we deduce the desired statement. 
\end{proof}

\subsubsection{General maximal laminations}
Lastly, we describe the natural finite approximation process that defines the shear-bend cocycle in general, extending the previous case: Let $\lambda\subset\mb{H}^2$ be an arbitrary maximal lamination. As before, let $P,Q\subset\mb{H}^2-\lambda$ be a pair of plaques and let $\mc{P}_{PQ}$ be the set of plaques separating $P$ from $Q$. Let 
\[
\mc{P}=\{P_1,\cdots,P_m\}\subset\mc{P}_{PQ}
\]
be a finite subset of plaques ordered from $P$ to $Q$. Any two consecutive $P_j,P_{j+1}$ cobound a (possibly empty) region $U_j$. We decompose its boundary as $\partial U_j=\ell_j\cup\ell_{j+1}$ with $\ell_j\subset\partial P_j$ and $\ell_{j+1}\subset\partial P_{j+1}$. We add to the finite collection $\mc{P}$ of plaques the triangles 
\[
\Delta(\ell_j^+,\ell_j^-,\ell_{j+1}^+),\Delta(\ell_j^-,\ell_{j+1}^-,\ell_{j+1}^+)
\]
obtaining a chain of triangles $P=T_1,T_2,\cdots,T_{3m-2},T_{3m-1}=Q$. 

We define
\[
\sigma_\rho(P,Q):=\sum_{j=1}^{3m-2}{\sigma^{\mb{B}}(T_j,T_{j+1})}.
\]

We then carefully choose an exhaustion $\{\mc{P}_n\}_{n\in\mb{N}}$ of $\mc{P}_{PQ}$ by an finite subsets and we set
\[
\sigma_\rho(P,Q):=\lim_{n\to\infty}{\sigma^{\mb{B}}_{\mc{P}_n}(P,Q)}.
\]

The existence of the limit as well as the independence of the choices made to define it and the fact that the limit object is a $\mb{B}$-transverse cocycle are proved in \cite{MV22a}:

\begin{thm}[Theorem B of \cite{MV22a}]
\label{thm:shear-bend cocycle}
Let $\rho_{X,Y}$ be a Mess representation. For every maximal geodesic lamination $\lambda\in\mc{GL}$, the finite approximation process converges and defines a $\mb{B}$-transverse cocycle $\sigma_\rho\in\mc{H}(\lambda;\mb{B})$.
\end{thm}

\begin{dfn}[Intrinsic Shear-Bend II]
Let $\rho_{X,Y}$ be a Mess representation. Let $\lambda$ be a maximal lamination. The cocycle $\sigma_\rho\in\mc{H}(\lambda;\mb{B})$ provided by Theorem \ref{thm:shear-bend cocycle} is the {\em intrinsic shear-bend cocycle} of the pleated set ${\hat S}_\lambda$.
\end{dfn} 

The following is a summary of results in Sections 4, 5 and 6 of \cite{MV22a}.

\begin{pro}
\label{thm:continuous extension}
We have the following properties:
\begin{enumerate}[(i)]
\item{If $\lambda$ has countably many leaves the definitions I and II coincide.}
\item{$(\sigma_\rho+{\bar\sigma}_\rho)/2$ is the shear cocycle of the intrinsic hyperbolic structure $Z_\lambda\in\T$.}
\item{The map $\lambda\in\mc{GL}_m\to\sigma_\rho\in\mc{W}(\tau;\mb{B})$ is continuous with respect to the Hausdorff topology on $\mc{GL}_m$. Here $\mc{W}(\tau;\mb{R})$ is the weight space of a train track $\tau$ carrying $\lambda$.}
\end{enumerate} 
\end{pro}

\subsection{Gauss map}
In order to prove Theorem \ref{thm:pleated surfaces ads} we study the {\em Gauss map} of the pleated set ${\hat S}_\lambda$ which we now describe. To this purpose let us begin with some general observations.

The group ${\rm PSL}_2(\mb{R})\times{\rm PSL}_2(\mb{R})$ acts transitively on oriented timelike geodesics. The stabilizer of $\gamma(t)=\cos(t)I+\sin(t)J\in{\rm PSO}(2)$ where $J={\tiny \left(\begin{array}{c c} 0 &-1\\ 1 &0\\ \end{array}\right)}$ is ${\rm PSO}(2)\times{\rm PSO}(2)$. 

Therefore, the space of oriented timelike geodesics is naturally ${\rm PSL}_2(\mb{R})\times{\rm PSL}_2(\mb{R})$-equivariantly identified with
\[
{\rm PSL}_2(\mb{R})/{\rm PSO}(2)\times{\rm PSL}_2(\mb{R})/{\rm PSO}(2) \cong \mathbb{H}^2 \times \mathbb{H}^2 ,
\] 
that is, the symmetric space of ${\rm PSL}_2(\mb{R})\times{\rm PSL}_2(\mb{R})$. We identify $\mb{RP}^1$ with $\mb{P}\{A\in M_2(\mb{R})\left|\;{\rm rk}(A)=1\right.\}/{\rm PSO}(2)$ and $\mb{H}^2$ with ${\rm PSL}_2(\mb{R})/{\rm PSO}(2)$.

\begin{lem}
\label{lem:gauss map}
Let $H\subset\mb{H}^{2,1}$ be a (oriented) spacelike plane. Consider the map $g=(g_l,g_r):H\to\mb{H}^2\times\mb{H}^2$ where $g(x)$ is the future pointing timelike geodesic orthogonal to $H$ at $x$. Then $g_j$ is isometric and extends continuously to the map $g_j:\partial H\subset\mb{RP}^1\times\mb{RP}^1\to\mb{RP}^1$ sending $g_j(a_l,a_r)=a_j$ for $j=l,r$.
\end{lem}

\begin{proof}
By equivariance it is enough to check the claim for a specific hyperplane $H\subset\mb{H}^{2,1}={\rm PSL}_2(\mb{R})$. We choose $H$ to be the dual plane of $I$, that is $H=\mb{P}\{M\in{\rm SL}_2(\mb{R})\left|\;{\rm tr}(M)=0\right.\}$. As above, let $\gamma={\rm PSO}(2)$.

Notice that $J=H\cap \gamma$ and, hence, $g(J)=\gamma=([I],[I])$. As the diagonal group of ${\rm PSL}_2(\mb{R})\times{\rm PSL}_2(\mb{R})$ preserves $H$ and acts transitively on it, by equivariance we have $g_j(AJ)=[A]$, that is, the components $g_j$ are the restrictions of the standard projection $\pi:{\rm PSL}_2(\mathbb{R})\to{\rm PSL}_2(\mathbb{R})/{\rm PSO}(2)$ to $H$. Also observe that, as $\gamma$ is orthogonal to $H$ at $J$, the differential ${\rm d}\pi_J$ is isometric. Thus, by equivariance, ${\rm d}\pi$ is isometric everywhere. 

The boundary of $H$ is $\partial H=\mb{P}\{M\in M_2(\mb{R})\left|\;{\rm tr}(M)=0,{\rm rk}(M)=1\right.\}$. Notice that, by Hamilton-Cayley, every $M\in M_2(\mb{R})$ satisfies $M^2-{\rm tr}(M)M+{\rm det}(M)=0$. Therefore, if $M\in\partial H$, then $M^2=0\Longleftrightarrow{\rm Im}(M)={\rm Ker}(M)$. The map $g_j(AJ)=[A]$ extends continuously to a map $\partial H\to\mb{RP}^1$ sending $g_j({\rm Im}(M),{\rm Ker}(M))=[{\rm Im}(M)]=[{\rm Ker}(M)]$.
\end{proof}

Let $\rho_{X,Y}$ be a Mess representation with limit curve $\Lambda_{X,Y}\subset\mb{RP}^1\times\mb{RP}^1$.

\begin{lem}
\label{lem:gauss shear}
Consider two ideal spacelike adjacent triangles $\Delta=\Delta(a,b,c)$ and $\Delta'=(c,b,a')$ sharing a common edge $[b,c]$ and with vertices ordered as $a<b<a'<c$ along $\Lambda_{X,Y}$. Let $g=(g_l,g_r):{\rm int}(\Delta)\cup{\rm int}(\Delta')\to\mb{H}^2\times\mb{H}^2$ be the map sending $x$ to the future pointing timelike normal $g(x)\in\mb{H}^2\times\mb{H}^2$. Then
\begin{align*}
\sigma(\Delta,\Delta')=\frac{\sigma_{\mb{H}^2}(g_l(\Delta),g_l(\Delta'))+\sigma_{\mb{H}^2}(g_r(\Delta),g_r(\Delta'))}{2}\\
\beta(\Delta,\Delta')=\frac{\sigma_{\mb{H}^2}(g_l(\Delta),g_l(\Delta'))-\sigma_{\mb{H}^2}(g_r(\Delta),g_r(\Delta'))}{2}
\end{align*}
where $\sigma_{\mb{H}^2}(\Delta_1,\Delta_2)$ denotes the hyperbolic elementary shear of the adjacent ideal triangles $\Delta_1,\Delta_2\subset\mb{H}^2$.
\end{lem}

\begin{proof}
Identify $\mb{BP}^1$ with $\mb{RP}^1\times\mb{RP}^1$. By Lemma \ref{lem:gauss map}, their left and right projections of $\Delta,\Delta'$ are the ideal triangles $g_j(\Delta)=\Delta(a_j,b_j,c_j),g_j(\Delta')=\Delta(c_j,b_j,a'_j)$ where $j=l,r$ respectively. Notice that we have $a_j<b_j<a_j'<c_j$ on $\mb{RP}^1$ because the set $\Delta\cup\Delta'$ is acausal. In particular, $$\sigma_{\mb{H}^2}(g_j(\Delta),g_j(\Delta')) = \log(-\beta^\mb{R}(b_j,c_j,a_j,a_j'))$$
by Lemma \ref{lem:elementary case}. Recall that $\sigma(\Delta,\Delta'),\beta(\Delta,\Delta')$ are the real and para-complex parts of $\sigma^\mb{B}(\Delta,\Delta')=\sigma^\mb{B}(a,b,c,d)$ and that, by definition, $\sigma^\mb{B}(b,c,a,a')=\log(-\beta^\mb{B}(b,c,a,a'))$. The conclusion follows from Lemma \ref{lem:para birapporto}.
\end{proof}

We are ready to prove Theorem \ref{thm:pleated surfaces ads}.

\subsection{The proof of Theorem \ref{thm:pleated surfaces ads}}

Let $\rho_{X,Y}$ be a Mess representation. 

Consider the pleated set ${\hat S}_\lambda$ associated with the maximal lamination $\lambda$. Every point $x\in{\hat S}_\lambda-{\hat \lambda}$ lies in a plaque and, therefore, has a well-defined future pointing timelike unit normal direction $g(x)$. The map $g=(g_l,g_r):{\hat S}_\lambda-{\hat \lambda}\to\mb{H}^2\times\mb{H}^2$ is the {\em Gauss map} of the pleated set ${\hat S}_\lambda$. By Lemma \ref{lem:gauss map}, it is $\rho_{X,Y}$-equivariant and totally geodesic on each plaque. 

\begin{proof}[Proof of Theorem \ref{thm:pleated surfaces ads}]
We split the proof into two cases.

{\em Maximal laminations with countably many leaves}. Let $P,Q$ be distinct plaques. By definition and by Lemma \ref{lem:elementary case}, it is enough to consider the case where $P,Q$ are either adjacent or asymptotic to the same leaf. The claim then follows from the computations of Lemmas \ref{lem:gauss map} and \ref{lem:gauss shear}.

{\em General maximal laminations}. The general case follows density of finite leaved maximal laminations in $\mc{GL}_m$ and continuity properties of the cocycles as given in Theorem \ref{thm:continuous extension}.
\end{proof} 

\subsection{Shear-bend parametrization}
The proof of Theorem \ref{thm:shear-bend ads} is a combination of Theorem \ref{thm:pleated surfaces ads} and some properties of the classical shear coordinates $\Phi:\T\to\mc{H}(\lambda;\mb{R})$.

\begin{proof}[Proof of Theorem \ref{thm:shear-bend ads}]
We have 
\[
\mc{H}(\lambda;\mb{B})=\frac{1+\tau}{2}\mc{H}(\lambda;\mb{R})\oplus\frac{1-\tau}{2}\mc{H}(\lambda;\mb{R})
\]
as $\mb{B}$-modules.

{\em Part (1)}. Recall that $\sigma^\mb{B}_\lambda=\sigma+\tau\beta$ and that, by Theorem \ref{thm:pleated surfaces ads}, we have $\sigma=(\sigma_\lambda^X+\sigma_\lambda^Y)/2$ and $\beta=(\sigma_\lambda^X-\sigma_\lambda^Y)/2$. Therefore, in terms of the above splitting, the shear-bend map decomposes as
\[
\Psi:\rho_{X,Y}\to\sigma^\mb{B}_\lambda=\frac{1+\tau}{2}\sigma_\lambda^X\oplus\frac{1-\tau}{2}\sigma_\lambda^Y.
\]
The single components $\Phi(X),\Phi(Y)=\sigma_\lambda^X,\sigma_\lambda^Y$ are analytic by Theorem \ref{thm:thurston bonahon}. Injectivity also follows from the injectivity in the same theorem since: 
\begin{align*}
\sigma_\rho^\mb{B}=\sigma_{\rho'}^\mb{B}\Longleftrightarrow\sigma_\lambda^X=\sigma_\lambda^{X'}\text{ and }\sigma_\lambda^Y=\sigma_\lambda^{Y'}.
\end{align*}

It remains to be checked that the map respects the para-complex structures of $\T\times\T$ and $\mc{H}(\lambda;\mb{B})$. The para-complex structure $\mb{J}$ acts on $T_X\T\oplus T_Y\T$ simply as $\mb{J}(u,v)=(u,-v)$ and acts on $\mc{H}(\lambda;\mb{B})$ as the multiplication by $\tau$. Denote by $\Phi:\T\to\mc{H}(\lambda;\mb{R})$ the classical shear coordinates, we have:
\begin{align*}
{\rm d}\Psi\mb{J}(u,v) &={\rm d}\Psi(u,-v)=\frac{1+\tau}{2}{\rm d}\Phi(u)\oplus\frac{1-\tau}{2}(-{\rm d}\Phi(v))\\
&=\tau\left(\frac{1+\tau}{2}{\rm d}\Phi(u)\oplus\frac{1-\tau}{2}{\rm d}\Phi(v)\right)=\tau {\rm d}\Psi(u,v).
\end{align*}

{\em Part (2)}. The Thurston's symplectic form on $\mc{H}(\lambda;\mb{B})$ splits as 
\[
\omega_{{\rm Th}}^\mb{B}=\frac{1+\tau}{2}\omega_{{\rm Th}}^\mb{R}\oplus\frac{1-\tau}{2}\omega_{{\rm Th}}^\mb{R},
\]
with respect to the above decomposition. Thus, by Theorem \ref{thm:symplectic length}, we have
\begin{align*}
\omega_{{\rm Th}}^\mb{B}(\sigma_\rho^\mb{B},\mu) &=\frac{1+\tau}{2}\omega_{{\rm Th}}^\mb{R}(\sigma_\lambda^X,\mu)+\frac{1-\tau}{2}\omega_{{\rm Th}}^\mb{R}(\sigma_\lambda^Y,\mu)\\
&=\frac{1+\tau}{2}L_X(\mu)+\frac{1-\tau}{2}L_Y(\mu)\\
&=\frac{L_X(\mu)+L_Y(\mu)}{2}+\tau\frac{L_X(\mu)-L_Y(\mu)}{2}\\
\end{align*}
for every $\mu\in\mc{ML}_\lambda$. We will see in the next section that $L_\rho=(L_X+L_Y)/2$ and $\theta_\rho=(L_X-L_Y)/2$.

{\em Part (3)}. By part (1), the image of $\Psi$ is 
\[
\left\{\sigma+\tau\beta\in\mc{H}(\lambda;\mb{B})\left|\;\sigma+\beta,\sigma-\beta\in\T\subset\mc{H}(\lambda;\mb{R})\right.\right\}.
\]
By Theorem \ref{thm:thurston bonahon}, we have
\[
\T=\{\sigma\in\mc{H}(\lambda,\mb{R})\left|\;\omega_{{\rm Th}}(\sigma,\bullet)>0\text{ \rm on }\mc{ML}_\lambda\right.\},
\]
Thus 
\begin{align*}
&\sigma+\tau\beta\in\Psi(\T\times\T)\Leftrightarrow\omega_{{\rm Th}}(\sigma\pm\beta,\mu)>0\\
&\Leftrightarrow\omega_{{\rm Th}}(\sigma,\mu)^2-\omega_{{\rm Th}}(\beta,\mu)^2=\left|\omega_{{\rm Th}}^\mb{B}(\sigma+\tau\beta,\mu)\right|^2_{\mb{B}}>0
\end{align*}
for every $\mu\in\mc{ML}_\lambda$.

{\em Part (4)}. By work of Bonahon and Sözen \cite{BS01}, we have that $\Phi^*\omega_{{\rm Th}}=c\cdot\omega_{{\rm WP}}$. The conclusion comes from the fact that $\Psi$ splits as $\frac{1+\tau}{2}\Phi\oplus\frac{1-\tau}{2}\Phi$ and $\omega_{{\rm Th}}^\mb{B}$ splits as $\frac{1+\tau}{2}\omega_{{\rm Th}}^\mb{R}\oplus\frac{1-\tau}{2}\omega_{{\rm Th}}^\mb{R}$.
\end{proof}

\section{Length functions in anti de Sitter 3-manifolds}
\label{sec:sec5}

In this section we study the anti de Sitter length functions associated with Mess representations and prove Theorem \ref{thm:length}.

\subsection{Moving endpoints orthogonally}
Let us start with some estimates in $\mb{H}^{2,1}$ on how the length of a spacelike segment changes if we move its endpoints orthogonally in timelike directions. The following is an elementary computation:

\begin{lem}
\label{lem:move orthogonally}
Let $[x,y]$ be a spacelike segment. Let $v\in T_x\mb{H}^{2,1},w\in T_y\mb{H}^{2,1}$ be unit timelike vectors orthogonal to $[x,y]$. Consider $p=\cos(t)x+\sin(t)v$ and $q=\cos(t)y+\sin(t)w$. Then
\begin{enumerate}
\item{$[v,w]$ lies on the dual geodesic of $[x,y]$. Hence, it is spacelike.}
\item{We have 
\[
-\langle p,q\rangle=\cos(t)^2\cosh(d_{\mb{H}^{2,1}}(x,y))+\sin(t)^2\cosh(d_{\mb{H}^{2,1}}(v,w)).
\]
As $-\langle p,q\rangle>1$, $[p,q]$ is spacelike and $\cosh(d_{\mb{H}^{2,1}}(p,q))=-\langle p,q\rangle$.}
\end{enumerate}
\end{lem}

In order to manipulate better some inequalities, later on we will use several times the following estimates on hyperbolic trigonometric functions: 

\begin{lem}
\label{lem:analysis}
We have:
\begin{enumerate}
\item{For every $a_0 > 0$ and $b_0 \in (0,\frac{\pi}{2})$, there exists $\kappa>0$ such that
\[
\cos(b)^2\cosh(a)+\sin(b)^2\cosh(a-a_0)\le\cosh(a-\kappa)
\]
for every $a \in [a_0,+\infty)$ and $b \in [b_0,\pi/2]$.}
\item{For every $a_0 > 0$ there exists $c_0 \in (0,1)$ such that
\[
c \, \cosh(a)\ge\cosh(a-\eta(c))
\]
for all $a \ge a_0$ and $c\in[c_0,1]$, where $\eta(c) = \mathrm{arccosh}(1/c)$.}
%\item{For every $a\geq1$ and $b > 0$ we have $\cosh(ab)\ge a\cosh(b)$.}
\end{enumerate} 
\end{lem}

\begin{proof} 
A straightforward computation shows that, for every $u>0$, the function $x \mapsto \cosh(x-u)/\cosh(x)$ is strictly decreasing over $\mathbb{R}$. In particular for every $x \geq x_0 > 0$ we have
\[
e^{- u} < \frac{\cosh(x - u)}{\cosh(x)} \leq \frac{\cosh(x_0 - u)}{\cosh(x_0)} .
\]

\vspace{1em}

\noindent {\em Inequality (1)}.  We first rewrite the desired statement as
\[
\cos(b)^2+\sin(b)^2\frac{\cosh(a-a_0)}{\cosh(a)}\le\frac{\cosh(a-\kappa)}{\cosh(a)}.
\]

As $x \mapsto \cosh(x-a_0)/\cosh(x)$ is decreasing and $b \geq b_0 \in (0,\frac{\pi}{2}]$, we have
\begin{align*}
	\cos(b)^2+\sin(b)^2\frac{\cosh(a-a_0)}{\cosh(a)} & \le \cos(b)^2+\sin(b)^2\frac{1}{\cosh(a_0)} \\
	& = 1 - \sin(b)^2 \left(1-\frac{1}{\cosh(a_0)}\right) \\
	& \leq 1 - \sin(b_0)^2 \left(1-\frac{1}{\cosh(a_0)}\right) \\
	& = \cos(b_0)^2+\sin(b_0)^2\frac{1}{\cosh(a_0)} <1.
\end{align*}
Since $\cosh(a - \kappa)/\cosh(a) > e^{-\kappa}$, it is enough to choose $\kappa>0$ so that $\cos(b_0)^2+\sin(b_0)^2/\cosh(a_0)<e^{-\kappa}$.
\vspace{1em}

\noindent {\em Inequality (2).}  We write $c = 1/\cosh(\delta)$ for some $\delta \geq 0$. For every $\delta \in [0,a_0]$ and for every $a \geq a_0$ we have
\begin{align*}
	\frac{\cosh(a - \delta)}{c} & = \cosh(\delta) \cosh(a - \delta) \\
	& \leq \cosh(\delta) \cosh(a - \delta) + \sinh(\delta) \sinh(a - \delta) \\
	& = \cosh(a) .
\end{align*}
Hence the assertion follows if we set $c_0 : = 1/\cosh(a_0)$.
\end{proof}

\subsection{Length and pleated surfaces}
We now introduce loxodromic transformations of $\mb{H}^{2,1}$ and the length functions associated to Mess representations.

\begin{dfn}[Loxodromic]
An isometry $\gamma=(A,B)\in{\rm PSL}_2(\mb{R})\times{\rm PSL}_2(\mb{R})$ is {\em loxodromic} if $A,B$ are both loxodromic transformations of ${\rm PSL}_2(\mb{R})$. A loxodromic transformation $\gamma$ preserves two disjoint (dual) lines 
\[
\ell=[(x_A^+,x_B^+),(x_A^-,x_B^-)],\ell^*=[(x_A^+,x_B^-),(x_A^-,x_B^+)]\subset\mb{H}^{2,1},
\]
where $x_A^\pm,x_B^\pm$ are the attracting and repelling fixed points of $A,B$ on $\mb{RP}^1$, and acts on them by translations by 
\[
L(\gamma)=\frac{L(A)+L(B)}{2}\;\text{ \rm and }\;\theta(\gamma)=\frac{\left|L(A)-L(B)\right|}{2}
\]
respectively where $L(A),L(B)$ are the translation lengths of $A,B$. The quantities $L(\gamma)$ and $\theta(\gamma)$ are the {\em translation length} and {\em torsion} of $\gamma$. 
\end{dfn}

Notice that if $\rho_{X,Y}$ is a Mess representation, then for every $\gamma\in\Gamma-\{1\}$ the transformation $\rho_{X,Y}(\gamma)=(\rho_X(\gamma),\rho_Y(\gamma))$ is loxodromic because $\rho_X,\rho_Y$ are holonomies of hyperbolic structures. Furthermore, as $\Lambda_{X,Y}\subset\mb{RP}^1\times\mb{RP}^1$ is the graph of the unique $(\rho_X-\rho_Y)$-equivariant homeomorphism $h_{X,Y}:\mb{RP}^1\times\mb{RP}^1$, we see that the axis $\ell_\gamma$ of $\rho_{X,Y}(\gamma)$, having the endpoints on $\Lambda_{X,Y}$, is contained in $\mc{CH}_{X,Y}$.

We are now ready to prove the first part of Theorem \ref{thm:length}.

\begin{pro}
\label{pro:lengthA}
Let $\rho_{X,Y}$ a Mess representation. Let $\gamma\in\Gamma-\{1\}$ be a non-trivial element, denote by $\ell\subset\mc{CH}_{X,Y}$ the axis of $\rho_{X,Y}(\gamma)$. Let $\lambda\subset\Sigma$ be a maximal lamination, let $Z_\lambda\in\T$ be the intrinsic hyperbolic structure on ${\hat S}_\lambda/\rho_{X,Y}(\Gamma)$ where ${\hat S}_\lambda\subset\mc{CH}_{X,Y}$ is the pleated set associated with $\lambda$. Let $\delta$ be the maximal timelike distance of $\ell$ from ${\hat S}_\lambda$. Then:
\[
\cosh(L_Z(\gamma))\le\cos(\delta)^2\cosh(L_\rho(\gamma))+\sin(\delta)^2\cosh(\theta_\rho(\gamma)).
\]
\end{pro}

\begin{proof}
Let $x\in\ell,y\in{\hat S}_\lambda$ be points that realize the maximal timelike distance $\delta$. Notice that the timelike segment $[x,y]$ is orthogonal to $\ell$ at $x$. Denote by $v\in T_x\mb{H}^{2,1}$ the unit timelike vector tangent to $[x,y]$. We can write $y=\cos(\delta)x+\sin(\delta)v$. 

We now apply Lemma \ref{lem:move orthogonally} to the spacelike segment $[x,\rho_{X,Y}(\gamma)x]\subset\ell$ and the timelike unit tangent vectors $v,\rho_{X,Y}(\gamma)v$. We have:
\begin{align*}
&\cosh\left(d_{\mb{H}^{2,1}}(y,\rho_{X,Y}(\gamma)y)\right)\\
&=\cos(\delta)^2\cosh\left(d_{\mb{H}^{2,1}}(x,\rho_{X,Y}(\gamma)x)\right)+\sin(\delta)^2\cosh\left(d_{\mb{H}^{2,1}}(v,\rho_{X,Y}(\gamma)v)\right).
\end{align*}

Notice that $d_{\mb{H}^{2,1}}(x,\rho_{X,Y}(\gamma)x)=L_\rho(\gamma)$ and $d_{\mb{H}^{2,1}}(v,\rho_{X,Y}(\gamma)v)=\theta_\rho(\gamma)$.

The conclusion then follows from Proposition \ref{pro:hyperbolic structure} which says that the intrinsic hyperbolic distance between $y,\rho_{X,Y}(\gamma)y$ on ${\hat S}_\lambda$ is smaller than $d_{\mb{H}^{2,1}}(y,\rho_{X,Y}(\gamma)y)$ and the fact that $L_Z(\gamma)$ coincides with the minimal displacement of $\rho_{X,Y}(\gamma)$ with respect to the hyperbolic metric on ${\hat S}_\lambda$.
\end{proof}

\subsection{Intersection and pleated surfaces}
We then prove the second part of Theorem \ref{thm:length}.

\begin{pro}
\label{pro:lengthB}
Let $\rho_{X,Y}$ be a Mess representation. Let $\gamma\in\Gamma-\{1\}$ be a non-trivial element, denote by $\ell\subset\mc{CH}_{X,Y}$ the axis of $\rho_{X,Y}(\gamma)$. Let $\delta^\pm$ be the maximal timelike distance of $\ell$ from $\lambda^\pm$. Then:
\[
\cosh(i(\lambda^\pm,\gamma))\le\sin(\delta^\pm)^2\cosh(L_\rho(\gamma))+\cos(\delta^\pm)^2\cosh(\theta_\rho(\gamma)).
\]

\end{pro}

\begin{proof}
Let $[x,x^\pm]$ be a timelike segment, with $x\in\ell,x^\pm\in\ell^\pm\subset\lambda^\pm$ that realizes the maximal timelike distance $\delta^\pm$. Notice that $[x,x^\pm]$ is orthogonal to both $\ell,\ell^\pm$. Let $v\in T_x\mb{H}^{2,1},v^\pm\in\T_{x^\pm}\mb{H}^{2,1}$ be the unit speed timelike vectors tangent to the geodesic $[x,x^\pm]$ at the endpoints. 

\begin{claim}{1}
\label{claim1'}
We have 
\[
\cosh(d_{\mb{H}^{2,1}}(v^\pm,\rho_{X,Y}(\gamma)v^\pm))=\sin(\delta^\pm)^2\cosh(L_\rho(\gamma))+\cos(\delta^\pm)^2 \cosh(\theta_\rho(\gamma)).
\]
\end{claim}

\begin{proof}[Proof of the claim]
Note that 
\[
v^\pm=-\cos(\pi/2-\delta^\pm)x+\sin(\pi/2-\delta^\pm)v
\]
and that $v$ and $\rho_{X,Y}(\gamma)v$ are both orthogonal to the segment $[x,\rho_{X,Y}(\gamma)x]\subset\ell$. The claim follows from Lemma \ref{lem:move orthogonally}.
\end{proof}

\begin{claim}{2}
\label{claim2'}
Let $v,v',v''\in\mb{H}^{2,1}$ be dual to the supporting pairwise distinct planes $H,H',H''$ of $\partial^+\mc{CH}_{X,Y}$.
\begin{enumerate}[(i)]
\item{If  $v,v',v''$ all lie on a common minimizing path inside $\mc{S}^+$, then $H\cap H'\cap H''$ is either empty or equal to a line. The latter happens if and only if $v, v',v''$ lie on a spacelike geodesic segment of $\mathbb{H}^{2,n}$.}
\item{If $v,v',v''$ all lie on a common minimizing path inside $\mc{S}^+$ and $v < v' < v''$, then the reverse triangle inequality holds
\[
d_{\mb{H}^{2,1}}(v,v'')\ge d_{\mb{H}^{2,1}}(v,v')+d_{\mb{H}^{2,1}}(v',v'').
\]
}
\end{enumerate}
\end{claim}

\begin{proof}[Proof of the claim]
Consider the faces $F,F',F''=H,H',H''\cap\partial^+\mc{CH}_{X,Y}$. As $\mc{S}^+$ is an $\mb{R}$-tree, there are two possibilities: Either one of the faces separates the other two on $\partial^+\mc{CH}_{X,Y}$ or there is a unique face $G\subset\partial^+\mc{CH}_{X,Y}$ different from $H, H', H''$ that separates every pair of them. The first case corresponds to the configuration where the dual points $v,v',v''$ lie on a minimizing path inside $\mc{S}^+$. The second case corresponds to the configuration where $v,v',v''$ are the vertices of a tripod in $\mc{S}^+$ with center $w$, the dual point of $G$. Let us consider the first case. In addition, let us assume that $v < v' < v''$ without loss of generality. Then either the lines $H\cap H'$, $H'\cap H''$, and $H \cap H''$ coincide, or $F'$ separates $H\cap H'$ from $H'\cap H''$ in $H'$. Hence the triple intersection $H\cap H'\cap H''$ is either empty or equal to a line.

The second part of the claim follows from Lemma 6.3.5 of \cite{BB09}.
\end{proof}

\begin{claim}{3}
\label{claim3'}
Let $v,w\in\mc{S}^+$ be distinct points. Then
\[
d_{\mc{S}^+}[v,w]\le d_{\mb{H}^{2,1}}(v,w).
\]
\end{claim}

\begin{proof}[Proof of the claim]
Let $\alpha:I=[0,1]\to\mc{S}^+$ be an injective weakly regular path joining $v$ and $w$. By Lemma \ref{lem:path length}, we have
\[
L=\int_I{\vert{\dot \alpha}(t)\vert{ \rm dt}}=\lim_{\ep\to 0}\int_{[0,1-\varepsilon]}{\frac{d_{\mb{H}^{2,1}}(\alpha(t),\alpha(t+\ep))}{\ep}{ \rm dt}}.
\]

If $\ep<\ep_0$ then
\[
\left|\int_I{\frac{d_{\mb{H}^{2,1}}(\alpha(t),\alpha(t+\ep))}{\ep}{ \rm dt}}-L\right|<\delta
\]

Choose $\ep=1/2^k$. For convenience, we take dyadic approximations of the integral with Riemann sums: 
\[
\int_I{\frac{d_{\mb{H}^{2,1}}(\alpha(t),\alpha(t+1/2^k))}{1/2^k}{ \rm dt}}=\lim_{n\to\infty}{\sum_{p=0}^{2^n-2^{n-k}}{\frac{d_{\mb{H}^{2,1}}(\alpha(p/2^n),\alpha(p/2^n+1/2^k))}{1/2^k}}\cdot\frac{1}{2^n}}.
\]

We reorganize the sum as
\begin{align*}
\frac{2^k}{2^n} \cdot \sum_{p=0}^{2^n - 2^{n-k}}& d_{\mb{H}^{2,1}}(\alpha(\tfrac{p}{2^n}),\alpha(\tfrac{p}{2^n}+\tfrac{1}{2^k})) = \\
&=\frac{2^k}{2^n} \cdot \sum_{j=0}^{2^{n-k} - 1}\left(\sum_{q=0}^{2^k-2}{d_{\mb{H}^{2,1}}(\alpha(\tfrac{j}{2^n}+\tfrac{q}{2^k}),\alpha(\tfrac{j}{2^n}+\tfrac{q+1}{2^k}))}\right) \\
&\qquad + \frac{2^k}{2^n} \, d_{\mathbb{H}^{2,1}}(\alpha(1 - \tfrac{1}{2^k}), \alpha(1))  \\
&\leq\frac{2^k}{2^n} \cdot \sum_{j=0}^{2^{n-k} - 1} d_{\mathbb{H}^{2,1}}(\alpha(\tfrac{j}{2^n}), \alpha(\tfrac{j}{2^n} + 1 - \tfrac{1}{2^k})) \tag{Claim \ref{claim2'}} \\
&\qquad + \frac{2^k}{2^n}\, d_{\mathbb{H}^{2,1}}(\alpha(1 - \tfrac{1}{2^k}), \alpha(1)) \\
&\le\frac{2^k}{2^n}\cdot\sum_{j=0}^{2^{n-k} - 1}{d_{\mb{H}^{2,1}}\left(\alpha(0),\alpha(1)\right)}=d_{\mb{H}^{2,1}}\left(v,w\right) \tag{Claim \ref{claim2'}}.
\end{align*}
The assertion follows by taking the limits first as $n \to \infty$, and then as $k \to \infty$.
\end{proof}

We have:
\begin{align*}
\cosh(i(\lambda,\gamma)) &\le\cosh(d_{\mc{S}^+}(v,\rho_{X,Y}(\gamma) v)) \tag{Prop. \ref{pro:singularities}}\\
 &\le\cosh(d_{\mb{H}^{2,1}}(v,\rho_{X,Y}(\gamma) v)) \tag{Claim \ref{claim3'}}\\
 &=\sin(\delta^\pm)^2\cosh(L_\rho(\gamma))+\cos(\delta^\pm)^2\cosh(\theta_\rho(\gamma)) \tag{Claim \ref{claim1'}}.
\end{align*}
\end{proof}

\section{Length functions in Teichmüller space}
\label{sec:sec6}

In this section we carry out an anti de Sitter analysis of length function on Teichmüller space on both global and infinitesimal scales and prove Theorems \ref{thm:length convex} and \ref{thm:earthquakes}.

\subsection{Orthogonal projection to a line}
We begin with some explicit computations on the orthogonal projection $\pi:\mb{H}^{2,1}\to\ell$ to a spacelike geodesic.

\begin{lem}
\label{lem:point line}
Let $y,\ell$ be a point and a spacelike line in $\mb{H}^{2,1}$ such that the rays $[y,\ell^\pm]$ are spacelike. Then
\[
m_{y,\ell}=\min_{x\in\ell}\{-\langle y,x\rangle\})=\sqrt{\frac{2\langle y,\ell^+\rangle\langle y,\ell^-\rangle}{-\langle \ell^+,\ell^-\rangle}}
\]
and it is realized at the unique point 
\[
x=\frac{1}{\sqrt{-2\langle\ell^+,\ell^-\rangle}}\left(\sqrt{\frac{\langle y,\ell^-\rangle}{\langle y,\ell^+\rangle}}\ell^++\sqrt{\frac{\langle y,\ell^+\rangle}{\langle y,\ell^-\rangle}}\ell^-\right)\in\ell
\]
such that $[y,x]$ is orthogonal to $\ell$. 
\end{lem}

\begin{proof}
Write $\ell(t)=(e^t\ell^++e^{-t}\ell^-)/\sqrt{-2\langle\ell^+,\ell^-\rangle}$ and consider the function $f(t)=-\langle\ell(t),y\rangle$. As $[y,\ell^+],[y,\ell^-]$ are spacelike, we have $f(t)\to\infty$ as $|t|\to\infty$. Hence, $f(t)$ has a minimum which is a critical point. The unique critical point of the function is at $e^{2t}=\langle y,\ell^-\rangle/\langle y,\ell^+\rangle$. The conclusion follows by elementary computations.
\end{proof}

\subsection{Convexity of length functions}

We now describe the purely anti de Sitter proof of (strict) convexity of length functions on Teichmüller space $\T$ in shear coordinates for an arbitrary maximal lamination $\lambda\subset\Sigma$.

We prove separately the two parts of Theorem \ref{thm:length convex}.

\begin{pro}
\label{pro:loops}
Let $\lambda\subset\Sigma$ be a maximal lamination. Let $\gamma\in\Gamma-\{1\}$ be a non-trivial loop. The length function $L_\gamma:\T\subset\mc{H}(\lambda;\mb{R})\to(0,\infty)$ is convex. Moreover, convexity is strict if $\gamma$ intersects essentially every leaf of $\lambda$.
\end{pro}

\begin{proof}
Recall that a function $L:U\subset\mb{R}^n\to\mb{R}$ defined on an open convex subset $U\subset\mb{R}^n$ is (strictly) convex if and only if for every $x,y\in U$ we have a (strict) inequality
\[
L\left(\frac{x+y}{2}\right)\le\frac{L(x)+L(y)}{2}.
\]
Consider $X,Y\in\T$. Let $\rho_{X,Y}$ be the corresponding Mess representation. Let ${\hat S}_\lambda\subset\mc{CH}_{X,Y}$ be the pleated set associated with $\lambda$ and let $\rho_\lambda:\Gamma\to{\rm PSL}_2(\mb{R})$ be the holonomy of the intrinsic hyperbolic structure $Z_\lambda\in\T$ on ${\hat S}_\lambda/\rho_{X,Y}(\Gamma)$. By Theorem \ref{thm:pleated surfaces ads} we have $Z_\lambda=(X+Y)/2$ in $\mc{H}(\lambda;\mb{R})$. By Theorem \ref{thm:length}, we have 
\[
\cosh(L_{Z_\lambda}(\gamma))\le\cos(\delta)^2\cosh(L_\rho(\gamma))+\sin(\delta)^2\cosh(\theta_\rho(\gamma))
\]
where $\delta$ is the maximal timelike distance of the axis $\ell$ of $\rho_{X,Y}(\gamma)$ from the pleated set ${\hat S}_\lambda$. Notice that $\delta>0$ unless $\ell\subset{\hat S}_\lambda$ in which case $\ell$ does not intersect the bending locus. Also observe that unless $\rho$ is Fuchsian, which happens precisely when $X=Y$, the bending locus cannot be empty. Therefore, if $X,Y$ are distinct and the support of $\gamma\in\mc{C}$ intersects essentially every leaf of $\lambda$ we have $\delta>0$ and 
\[
\cos(\delta)^2\cosh(L_\rho(\gamma))+\sin(\delta)^2\cosh(\theta_\rho(\gamma))<\cosh(L_\rho(\gamma))
\]
as $L_\rho(\gamma)<\theta_\rho(\gamma)$. Since $\cosh(\bullet)$ is strictly increasing on $(0,\infty)$, we conclude $L_Z(\gamma)<L_\rho(\gamma)$.
\end{proof}

\begin{pro}
\label{pro:laminations}
Let $\lambda\subset\Sigma$ be a maximal lamination. Let $\gamma\in\mc{ML}$ be a measured lamination. The length function $L_\gamma:\T\subset\mc{H}(\lambda;\mb{R})\to(0,\infty)$ is convex. Furthermore, convexity is strict if the support of $\gamma$ intersects transversely each leaf of $\lambda$.
\end{pro}

\begin{proof}
We immediately deduce convexity by Proposition \ref{pro:loops} and density of weighted simple curves in $\mc{ML}$ and $\mc{C}^\infty$-convergence of length functions $L_{\gamma_n}\to L_\gamma$ if $\gamma_n\to\gamma$ in $\mc{ML}$.

We now discuss strict convexity. 

Consider $X,Y\in\T$ and the Mess representation $\rho:=\rho_{X,Y}$. Let ${\hat S}_\lambda\subset\mc{CH}_{X,Y}$ be the pleated set associated with $\lambda$. Let $\gamma\in\mc{ML}$ be a measured lamination whose support contains a leaf $\ell$ that intersects the bending locus of ${\hat S}_\lambda$ (which is non-empty unless the representation is Fuchsian).

Since $\ell$ intersects the bending locus, its geometric realization ${\hat \ell}$ is not contained on ${\hat S}_\lambda$. Let $x\in{\hat \ell}$ and $y\in{\hat S_\lambda}$ be points that realize the maximal timelike distance $\delta=\max\{\delta_{\mb{H}^{2,1}}(z,t)\left|\;z\in{\hat S}_\lambda,t\in{\hat \ell}\right.\}>0$.  

Let $K:=I\times J$ denote the neighborhood of $\ell$ in the space of geodesics $\mc{G}$ consisting of those lines with one endpoint in $I$ and another endpoint in $J$.

Recall that, by Lemma \ref{lem:point line}, we have
\[
m_{z,\ell}:=\min_{t\in\hat{\ell}}\{-\langle z,t\rangle\})=\sqrt{\frac{2\langle z,\hat{\ell}^+\rangle\langle z,\hat{\ell}^-\rangle}{-\langle \hat{\ell}^+,\hat{\ell}^-\rangle}}
\] 
and that the minimum is realized at a point $\pi(z)\in\hat{\ell}$, the orthogonal projection of $z$ to $\hat{\ell}$, described explicitly by
\[
\pi(z)=\frac{1}{\sqrt{-2\langle\hat{\ell}^+,\hat{\ell}^-\rangle}}\left(\sqrt{\frac{\langle z,\hat{\ell}^-\rangle}{\langle z,\hat{\ell}^+\rangle}}\hat{\ell}^+ + \sqrt{\frac{\langle z, \hat{\ell}^+\rangle}{\langle z,\hat{\ell}^-\rangle}}\hat{\ell}^-\right).
\]
As $y$ is connected to $\hat{\ell}$ by a timelike segment of length $\delta$ orthogonal to $\hat{\ell}$, we have $m_{y,\ell}=\cos(\delta)$. As $x\in\hat{\ell}$ we have $m_{x,\ell}=1$. By continuity of the above expressions, we have:

\begin{claim}{1}
\label{claim1} 
For every $\ep > 0$, there exist a neighborhood $K(\ep)=I(\ep) \times J(\ep)$ of $\ell\in\mc{G}$ and a neighborhood $U(\ep)$ of $x$ in $\mb{H}^{2,1}$ with the following properties:
\begin{enumerate}[(i)]
\item{$m_{y,\ell'}\in(\cos(2\delta),\cos(\delta/2))$ for every $\ell'\in K(\ep)$. In particular, $y$ is connected to every $\hat{\ell}'$ by a timelike segment of length at least $\delta/2$ and, hence, $\delta_{\mb{H}^{2,1}}(y,\hat{\ell}')\ge\delta/2$.}
\item{For every $\ell'\in K(\ep)$, $\hat{\ell}'$ intersects $U(\ep)$.}
\item{If $\ell_1,\ell_2\in K(\ep)$ are so that $\hat{\ell}_1\cup\hat{\ell}_2$ is acausal and $z_j\in\hat{\ell}_j\cap U(\ep)$, then $d_{\mb{H}^{2,1}}(z_1,z_2)<\ep$.}
\item{$m_{z,\ell'}\in(\cos(\ep),\cosh(\ep))$ for every $\ell'\in K(\ep)$ and $z\in U(\ep)$.}
\item{For every $z,w\in U(\ep)$ and $\ell'\in K(\ep)$, we have $d_{\mb{H}^{2,1}}(\pi'(z),\pi'(w))<\ep$ where $\pi'$ is the orthogonal projection onto $\hat{\ell}'$.}
\end{enumerate}
\end{claim}

Let $K(\ep)$ and $U(\ep)$ be the neighborhoods provided by the claim. 

As $\ell$ lies in the support of $\gamma$, we have $m(\ep):=\gamma(K(\ep)) > 0$. 

We approximate $\gamma$ in $\mc{ML}$ with a sequence of weighted simple closed curves $a_n\gamma_n$. By convergence of $a_n\gamma_n$ to $\gamma$, we have $a_n m_{n}(\ep) := a_n\gamma_n(K(\ep))\to m(\ep)$. Notice that $m_n = m_{n}(\ep)$ is the number of distinct leaves of the geometric realization ${\hat \gamma}_n$ contained in $K(\ep)$. Let $\ell_n$ be one of those leaves.

\begin{claim}{2}
\label{claim2}
There exists a constant $M > 1$ that depends only on the representation $\rho$ such that for any $\ep > 0$, $n \in \mathbb{N}$, and $\ell_n$ as above, we can find elements
\[
\alpha_{n,1},\cdots,\alpha_{n,m_{n}}\in\Gamma
\]
and corresponding points
\[
z_{n,0} < z_{n,1}<\cdots<z_{n,m_{n} - 1} < z_{n,m_n} = \rho(\gamma_n) z_{n,0}
\]
on ${\hat \ell}_n$ with the following properties: 
\begin{enumerate}[(i)]
\item{$\alpha_{n,m_{n}}\cdots\alpha_{n,1}=\gamma_n$.}
\item{$z_{n,0}, \rho(\alpha_{n,j}\cdots\alpha_{n,1})^{-1} z_{n,j} \in U(\ep/M)$ and $$d_{\mb{H}^{2,1}}(z_{n,j},\rho(\alpha_{n,j}\cdots\alpha_{n,1})z_{n,0})<\ep/M < \epsilon$$
for every $j \in \{1, \dots, m_n\}$.}
\item{The axis of $\alpha_{n,j}$ lies inside $\alpha_{n,j-1}\cdots\alpha_{n,1}(K(\ep))$.}
\end{enumerate}
\end{claim}

\begin{proof}[Proof of the claim] 
	
Let us start by applying Claim \ref{claim1} to an arbitrary $\ep' > 0$, and denote by
\[
\ell_n=\ell_{n,0},\cdots,\ell_{n,m_n - 1}
\]
the $m_n$ translates of $\ell_n$ contained in $K(\ep')$. We will later determine sufficient conditions on $\epsilon'$ (in term of $\epsilon$) that guarantee the desired properties.

By Claim \ref{claim1} part (ii), for every $j \in \{0, \dots, m_n - 1\}$ there exists some point $w_{n,j} \in \hat{\ell}_{n,j} \cap U(\ep')$. Since the leaves $\ell_{n,j}$ are in the same $\Gamma$-orbit, we can find elements $\beta_{n,j} \in \Gamma$ such that $z_{n,j}:=\rho(\beta_{n,j})\, w_{n,j}$ belongs to the spacelike geodesic segment $[w_{n,0}, \rho(\gamma_n) w_{n,0}] \subset \hat{\ell}_n$. Notice that $\beta_{n,0} = \mathrm{id} \in \Gamma$ and $z_{n,0} = w_{n,0}$. We also set $\beta_{n,m_n} := \gamma_n$, $z_{n,m_n} : = \rho(\gamma_n) w_{n,0}$. Up to reindexing the leaves $\ell_{n,j}$, we can assume that the points $z_{n,j}$ appear in linear order along $[w_{n,0}, \rho(\gamma_n) w_{n,0}]$, that is
\[
w_{n,0} = z_{n,0} < z_{n,1} < \dots < z_{n,m_n-1} < z_{n,m_n} = \rho(\gamma_n) w_{n,0} .
\]
For any $j \in \{1, \dots, m_n\}$, we then define $\alpha_{j,n} : = \beta_{n,j} \beta_{n,j-1}^{-1}$.

\vspace{1em}

\noindent {\bf Property (i).} The identity $\alpha_{n,m_n} \cdots \alpha_{n,1} = \beta_{n,m_n} = \gamma_n$ follows directly from our construction.

\vspace{1em}

\noindent {\bf Property (ii).} Notice that
\[
w_{n,0} \in U(\ep') \cap \hat{\ell}_{n,0}, \; w_{n,j} = \rho(\beta_{n,j})^{-1} z_{n,j} = \rho(\alpha_{n,j} \cdots \alpha_{n,1})^{-1} z_{n,j} \in U(\ep') \cap \hat{\ell}_{n,j} .
\]
Since the leaves $\ell_{n,0}$ and $\ell_{n,j}$ are lifts of a common simple closed curve in $S$, the union of their geometric realizations $\hat{\ell}_{n,0} \cup \hat{\ell}_{n,j}$ is acausal. In particular, it follows from Claim \ref{claim1} part (iii) that
\[
d_{\mathbb{H}^{2,1}}(\rho(\alpha_j \cdots \alpha_1) z_{n,0},z_{n,j}) = d_{\mathbb{H}^{2,1}}(w_{n,0}, \rho(\alpha_{n,j} \cdots \alpha_{n,1})^{-1} z_{n,j}) < \ep' .
\]

\vspace{1em}

\noindent {\bf Property (iii).} We deduce the last property from the stability of quasi-geodesics inside ${\hat S}_n$, the pleated set associated with the lamination $\lambda_n$ consisting of the closed geodesic $\gamma_n$ suitably completed to a maximal lamination of $\Sigma$ by adding finitely many leaves spiraling around $\gamma_n$. 

Let us explain how: For any $j \in \{1,\dots, m_n\}$, consider the concatenation of the translates 
\[
g_{n,j} = \bigcup_{k\in\mb{Z}} {\rho(\alpha_{n,j})^k\left([z_{n,j-1},z_{n,j}]_{{\hat S}_n}\cup[z_{n,j},\rho(\alpha_{n,j}) z_{n,j -1}]_{{\hat S}_n}\right)} ,
\]
where $[a,b]_{\hat{S}_n}$ denotes the length-minimizing path inside $\hat{S}_n$ between $a, b \in \hat{S}_n$. Notice that
\begin{align*}
	d_{\hat{S}_n}(z_{n,j},\rho(\alpha_{n,j}) z_{n,j -1}) & = d_{\hat{S}_n}(\rho(\beta_{n,j})^{-1} z_{n,j}, \rho(\beta_{n,j-1}^{-1}) \, z_{n,j -1}) \\
	& = d_{\hat{S}_n}(w_{n,j}, w_{n,j-1}) \\
	& \leq d_{\mathbb{H}^{2,1}}(w_{n,j}, w_{n,j-1}) \tag{$1$-Lipschitz dev. map} \\
	& < \epsilon' ,
\end{align*}
where in the last step we applied Claim \ref{claim1} part (iii) to $w_{n,j - 1} \in U(\ep') \cap \hat{\ell}_{n,j - 1}$, $w_{n,j} \in U(\ep') \cap \hat{\ell}_{n,j}$. By basic hyperbolic geometry, $g_{n,j}$ is a uniform quasi-geodesic on ${\hat S}_n$ with respect to the intrinsic hyperbolic metric, with quasi-geodesic constants that are $O(\ep')$-close to $1$ (see for example Section I.4.2 of \cite{CEG}). Hence, the invariant axis of $\rho(\alpha_{n,j})$ on ${\hat S}_n$ lies in the $O(\ep')$-neighborhood of $g_{n,j}$ with respect to the hyperbolic metric. In particular such endpoints are close to the endpoints of $(\alpha_{n,j -1} \cdots \alpha_{n,1})\ell_n$ on the Gromov boundary $\partial_\infty{\hat S}_n$.

Let $\phi_n:\partial_\infty{\hat S}_n\to\partial\Gamma$ be the unique $\Gamma$-equivariant homeomorphism. By Lemma \ref{lem:pleated cpt}, the hyperbolic structures ${\hat S}_n/\rho(\Gamma)$ lie in a compact subspace of Teichmüller space $\T$. Thus, as the boundary maps $\phi_n$ depend continuously on $S_n$, they are uniformly equicontinuous. In particular, it is not restrictive to assume that the function $O(\ep')$ is independent of the hyperbolic structure of $S_n$, the leaf $\ell_n$, and the weighted simple closed curves $a_n \gamma_n$. It follows that, if $\ep'$ is small enough (or, equivalently, if $M > 1$ is sufficiently large and $\ep' : = \ep / M$), the endpoints of the axis of $\alpha_j$ are contained in $K(\ep)$ for every $j \leq m_n$.
\end{proof}

Let $\alpha_j = \alpha_{n,j} \in \Gamma$ and $z_j = z_{n,j} \in \hat{\ell}_n$ be the elements provided by Claim \ref{claim3}, and define $x_j:=\rho(\alpha_j\cdots\alpha_1)x$ and $y_j=\rho(\alpha_j\cdots\alpha_1)y$, where $x_0 : = x \in \hat{\ell}$ and $y_0 : = y \in \hat{S}_\lambda$ maximize the timelike distance between $\hat{\ell}$ and $\hat{S}_\lambda$.

For any $j \in \{0,\dots, m_n - 1\}$, let $\delta_j:=\delta_{\mb{H}^{2,1}}(y_{j},\hat{\ell}_{\alpha_{j+1}})$ be the timelike distance of $y_j \in {\hat S}_\lambda$ from $\hat{\ell}_{\alpha_{j+1}}$, the axis of $\rho(\alpha_{j+1})$. By Property (iii) of Claim \ref{claim2}, we have $(\alpha_{j}\cdots\alpha_1)^{-1}\ell_{\alpha_{j+1}}\in K(\epsilon)$. Hence, Claim \ref{claim1} part (i) implies that 
\begin{equation}\label{eq:lower bound delta}
	\delta_j=\delta_{\mb{H}^{2,1}}(y_{j},\hat{\ell}_{\alpha_{j+1}})=\delta_{\mb{H}^{2,1}}(y,\rho(\alpha_{j}\cdots\alpha_1)^{-1}\hat{\ell}_{\alpha_{j+1}})>\delta/2.
\end{equation}

\begin{claim}{3}
\label{claim3} 
There exists $\kappa > 0$, depending only on $X, Y \in \T$ and on $\delta > 0$, such that for every $j \in \{0,\dots, m_n - 1\}$
\begin{align*}
\cosh(d_{\mb{H}^{2,1}}(y_{j},y_{j+1}))=\cos(\delta_j)^2\cosh(L_\rho(\alpha_{j+1}))+\sin(\delta_{j})^2\cosh(\theta_\rho(\alpha_{j+1}))\\
\le\cosh(L_\rho(\alpha_{j+1})-\kappa).
\end{align*}
\end{claim}

\begin{proof}[Proof of the claim]
Let $\pi_{j+1}(p)\in \hat{\ell}_{\alpha_{j+1}}$ be the unique point such that the segment $[p,\pi_{j+1}(p)]$ is orthogonal to $\hat{\ell}_{\alpha_{j+1}}$. Observe that $\pi_{j+1}(y_{j+1}) = \rho(\alpha_{j+1})\pi_{j+1}(y_j)$. Lemma \ref{lem:move orthogonally} applied to the spacelike segment $[\pi_{j+1}(y_j),\pi_{j+1}(y_{j+1})]$ of length $L_\rho(\alpha_{j+1})$ and the orthogonal timelike segments $$[y_j,\pi_{j+1}(y_j)], \quad \rho(\alpha_{j+1})[y_j,\pi_{j+1}(y_j)] = [y_{j+1},\rho(\alpha_{j+1})\pi_{j+1}(y_j)], $$
implies the first identity of our statement. To deduce the upper bound, we argue as follows: by relation \eqref{eq:lower bound delta} $\delta_j > \delta/2$ for every $j \in \{0, \dots, m_n - 1\}$ and
\[
L_\rho(\alpha_{j + 1}) - \theta_\rho(\alpha_{j + 1}) = L_Y(\alpha_j) \geq \min\{\mathrm{sys}(X), \mathrm{sys}(Y)\} =: r > 0 ,
\]
where $\rho = \rho_{X,Y}$ and $\mathrm{sys}(Z) > 0$ denotes the systole of the hyperbolic structure $Z \in \T$. In particular it follows that
\begin{align*}
	\cosh(d_{\mb{H}^{2,1}}(y_{j},y_{j+1})) =\cos(\delta_j)^2\cosh(L_\rho(\alpha_{j+1}))+\sin(\delta_{j})^2\cosh(\theta_\rho(\alpha_{j+1}))\\
	\le \cos(\delta_j)^2\cosh(L_\rho(\alpha_{j+1}))+\sin(\delta_{j})^2\cosh(L_\rho(\alpha_{j+1}) - r) .
\end{align*}
On the other hand, we have
\[
L_\rho(\alpha_{j+1}) = \frac{1}{2} (L_X(\alpha_{j + 1}) + L_Y(\alpha_{j + 1})) \geq r
\]
for every $j \in \{0, \dots, m_n - 1\}$. The conclusion now follows from Lemma \ref{lem:analysis} part (1) applied to $a_0 = r$, $b_0 = \delta/2$, $a = L_\rho(\alpha_{j+1})$ and $b = \delta_j$. Notice in particular that the resulting $\kappa$ depends only on $r$ and $\delta$, as desired.
\end{proof}

\begin{claim}{4}
\label{claim4}
There exist constants $C , \ep_0 > 0$ that depend only on the representation $\rho$ such that for every $\ep \in (0,\ep_0)$ and for every $j \in \{0, \dots, m_n - 1\}$ as in Claim \ref{claim2} we have
\[
L_\rho(\alpha_{j + 1})- d_{\mb{H}^{2,1}}(z_{j},z_{j+1})\le C \ep.
\]
\end{claim}

\begin{proof}[Proof of the claim] 
		As a first step, we show that
		$$|d_{\mathbb{H}^{2,1}}(\rho(\alpha_{j+1}) z_j, z_j) - d_{\mathbb{H}^{2,1}}(z_{j+1}, z_j)|\leq d_{\mathbb{H}^{2,1}}(\rho(\alpha_{j+1}) z_j, z_{j+1}) .$$ To see this, notice that $z_{j}$, $z_{j+1}$, and $\rho(\alpha_{j+1}) z_{j}$ lie on a common spacelike plane, since $z_{j}, z_{j+1} \in \hat{\ell}_n$ and $\rho(\alpha_{j+1}) z_{j}$ belongs to $\rho(\alpha_{j+1}) \hat{\ell}_n$, which are entirely contained in $\hat{S}_n$. Therefore, the inequality is a reformulation of the standard triangle inequality. Since both points $\rho(\alpha_{j+1}) z_{j}, z_{j+1}$ lie inside $\rho(\alpha_{j+1} \cdots \alpha_1) U(\ep/M)$, it follows from Claim \ref{claim2} part (ii) and Claim \ref{claim1} that
		$$|d_{\mathbb{H}^{2,1}}(\rho(\alpha_{j+1}) z_{j}, z_{j}) - d_{\mathbb{H}^{2,1}}(z_{j+1}, z_{j})| < \ep/M .
		$$
		We now prove that $d_{\mathbb{H}^{2,1}}(\rho(\alpha_{j+1}) z_{j}, z_{j}) > L_\rho(\alpha_{j+1}) - O(\ep)$. Consider the orthogonal projection $\pi_{j+1}$ onto $\hat{\ell}_{\alpha_{j+1}}$. By Claim \ref{claim2} part (ii) and \ref{claim1} part (iv), the quantity
		\[
		D_{j}=\min_{t\in\hat{\ell}_{\alpha_{j+1}}}\{-\langle z_{j},t\rangle\}=\sqrt{\frac{2\langle z_{j},\hat{\ell}_{\alpha_{j+1}}^+\rangle\langle z_{j},\hat{\ell}_{\alpha_{j+1}}^-\rangle}{-\langle\hat{\ell}_{\alpha_{j+1}}^+,\hat{\ell}_{\alpha_{j+1}}^-\rangle}}=-\langle z_{j},\pi_{j+1}(z_{j})\rangle.
		\]
		is contained in the interval $(\cos(\ep/M),\cosh(\ep/M))$.
	
%Let $\pi_j$ be the orthogonal projection to $\ell_{\alpha_j}$. By Claim \ref{claim1}, since $\ell_{\alpha_j}\in\alpha_{j-1}\cdots\alpha_1(K)$ and $x_{j-1},z_{j-1}\in\rho(\alpha_{j-1}\cdots\alpha_1)(U)$, we have that
%\[
%d_{\mb{H}^{2,1}}(\pi_j(x_{j-1}),\pi_j(z_{j-1}))\le\ep.
%\]
%Therefore $d_{\mb{H}^{2,1}}(\pi_j(x_{j-1}),\pi_j(x_j))-d_{\mb{H}^{2,1}}(z_{j-1},z_j)\le 2\ep$. 
%
%By Claim \ref{claim1}, we also have,
%\[
%D_j=\min_{t\in\ell_{\alpha_j}}\{-\langle x_{j-1},t\rangle\}=\sqrt{\frac{2\langle x_{j-1},\ell_{\alpha_j}^+\rangle\langle x_{j-1},\ell_{\alpha_j}^-\rangle}{-\langle\ell_{\alpha_j}^+,\ell_{\alpha_j}^-\rangle}}=-\langle x_{j-1},\pi_j(x_{j-1})\rangle.
%\]
%is contained in the interval $(\cos(\ep),\cosh(\ep))$. 

%We prove that $d_{\mb{H}^{2,1}}(\pi_j(x_{j-1}),\pi_j(x_j))\ge\ell(\alpha_j)-O(\ep)$.

If $D_j>1$, then the segment $[z_{j},\pi_{j+1}(z_{j})]$ is spacelike. Write $D_j=\cosh(d_j)$ and $z_{j}=\cosh(d_j)\,\pi_{j+1}(z_{j})+\sinh(d_j)\,v_j$ with $v_j$ orthogonal to $\hat{\ell}_{\alpha_{j+1}}$ at $\pi_{j+1}(z_{j})$. We have
\begin{align*}
&\cosh(d_{\mb{H}^{2,1}}(z_{j},\rho(\alpha_{j+1})z_{j}))=-\langle z_{j},\rho(\alpha_{j+1})z_{j}\rangle\\
&=\cosh(d_j)^2\cosh(L_\rho(\alpha_{j+1}))-\sinh(d_j)^2\cosh(\theta_\rho(\alpha_{j+1})).
\end{align*}

Hence, as $L_\rho(\alpha_{j+1})>\theta_\rho(\alpha_{j+1})$, we get $$\cosh(d_{\mb{H}^{2,1}}(z_{j},\rho(\alpha_{j+1})z_{j}))>\cosh(L_\rho(\alpha_{j+1})).$$

If $D_j<1$, then the segment $[z_{j},\pi_{j+1}(z_j)]$ is timelike. Write $D_j=\cos(d_j)$ and $z_{j}=\cos(d_j)\pi_{j+1}(z_{j})+\sin(d_j)v_j$ with $v_j$ orthogonal to $\hat{\ell}_{\alpha_{j+1}}$ at $\pi_{j+1}(z_{j})$. We have
\begin{align*}
&\cosh(d_{\mb{H}^{2,1}}(z_{j},\rho(\alpha_{j+1})z_{j}))=-\langle z_{j},\rho(\alpha_{j+1})z_{j}\rangle\\
&=\cos(d_j)^2\cosh(L_\rho(\alpha_{j+1}))+\sin(d_j)^2\cosh(\theta_\rho(\alpha_{j+1})).
\end{align*}
Thus
\[
\cosh(d_{\mb{H}^{2,1}}(z_{j},\rho(\alpha_{j+1})z_{j}))>\cos(d_j)^2 \cosh(L_\rho(\alpha_{j+1}))
\]
Let now $r > 0$ be the systole of the representation $\rho$. By part (2) of Lemma \ref{lem:analysis} applied to $a_0 = r$, there exists a constant $c_0 \in (0,1)$ such that
\[
\cos(d_j)^2\cosh(L_\rho(\alpha_{j+1})) \geq \cosh(L_\rho(\alpha_{j+1}) - \eta(\cos(d_j)^2)) ,
\]
for every $j$ that satisfies $\cos(d_j)^2 \geq c_0$, where $\eta(c) : = \mathrm{arccosh}(1/c)$. Since $d_j \leq \ep/M$, there exists a $\ep_0$ such that for any $\ep \leq \ep_0$ the condition $\cos(d_j)^2 \geq c_0$ holds for any $j$. A simple estimate shows that $\eta(c) \leq \frac{3}{c} \sqrt{1 - c}$ for every $c \in (0,1)$, implying that
\begin{align*}
	L_\rho(\alpha_{j+1}) - \eta(\cos(d_j)^2) & \geq L_\rho(\alpha_{j+1}) - \frac{3 \sin(d_j)}{\cos(d_j)^2} \\
	& \geq L_\rho(\alpha_{j+1}) - \frac{3 \sin(\ep/M)}{\cos(\ep/M)^2} \\
	& \geq L_\rho(\alpha_{j+1}) - C\ep ,
\end{align*}
for some constant $C > 0$ depending only on $\epsilon_0, M > 0$. Finally, combining the estimates obtained above, we conclude that
\begin{align*}
	\cosh(d_{\mb{H}^{2,1}}(z_{j},\rho(\alpha_{j+1})z_{j}))& > \cos(d_j)^2 \cosh(L_\rho(\alpha_{j+1})) \\
	& \geq \cosh(L_\rho(\alpha_{j+1}) - C\ep)
\end{align*}
for every $j \in \{0,\dots, m_n - 1\}$, which concludes the proof of the assertion. 
\end{proof}

Let $\kappa, \ep_0, C > 0$ be the constants provided by Claims \ref{claim3} and \ref{claim4} and choose $\ep = \frac{1}{2} \min\{\ep_0, \kappa/C\}$. Then for every $n \in \mathbb{N}$ we have
\begin{align*}
L_Z(\gamma_n) &\le d_{{\hat S}}(y,\gamma_ny) \\
 &\le\sum_j{d_{{\hat S}}(y_{j},y_{j+1})} \tag{Triangle inequality}\\
 &\le\sum_j{d_{\mb{H}^{2,1}}(y_{j},y_{j+1})} \tag{$1$-Lipschitz dev. map}\\
 &\le\sum_j{(L_\rho(\alpha_{j+1})-\kappa)} \tag{Claim \ref{claim3}}\\
 &\le\sum_j{(d_{\mb{H}^{2,1}}(z_{j},z_{j+1})+ C \ep - \kappa)} \tag{Claim \ref{claim4}}\\
 &\leq L_\rho(\gamma_n) - \frac{m_n \kappa}{2} .
\end{align*}
Multiplying by $a_n$ and taking the limit as $n \to \infty$, we deduce the desired assertion.
\end{proof}

\subsection{Second variation along earthquakes}
In the case of earthquakes, we make quantitative estimates and compute the second variation of length functions as given in Theorem \ref{thm:earthquakes}. 

%\begin{pro}
%\label{pro:second variation}
%Let $\lambda\in\mc{ML}$ be a measured lamination. Let $E_\lambda:[a,b]\to\T$ be an earthquake path driven by $\lambda$. Let $\gamma\in\Gamma-\{1\}$ be a non-trivial loop. Set $L_\gamma(t):=L_\gamma(E_\lambda(t))$. Then: 
%\[
%{\ddot L}_\gamma\ge\frac{1}{\sinh(L_\gamma)}\left|{\dot L}_\gamma\right|\left(i(\gamma,\lambda)-\left|{\dot L}_\gamma\right|\right).
%\]
%\end{pro}

\begin{proof}[Proof of Theorem \ref{thm:earthquakes}]
Let $Z_t:=E_\lambda(t)$, consider the Mess representation $\rho_t:=\rho_{Z_{-t},Z_t}$ with parameters $Z_{-t},Z_t\in\T$. Notice that, by Theorem \ref{thm:pleated surfaces ads}, we have $\lambda_t^+=t\lambda$ and $Z_{\lambda_t^+}=Z$ is constant. 

For convenience, we introduce $L_t:=L_{\rho_t}(\gamma)$ and $\theta_t:=\theta_{\rho_t}(\gamma)$.

By Propositions \ref{pro:lengthA} and \ref{pro:lengthB}, we have
\[
\cosh(L_Z(\gamma))\le\cos(\delta_t^\pm)^2\cosh(L_t)+\sin(\delta_t^\pm)^2\cosh(\theta_t)
\]
and
\[
\cosh(i(\lambda_t^\pm,\gamma))\le\sin(\delta_t^\pm)^2\cosh(L_t)+\cos(\delta_t^\pm)^2\cosh(\theta_t).
\]

Summing the inequalities, we get
\[
\cosh(t\cdot i(\lambda^+,\gamma))-\cosh(\theta_t)\le\cosh(L_t)-\cosh(L_Z(\gamma)).
\]
By the mean value theorem, we can write
\[
\cosh(t\cdot i(\lambda^+,\gamma))-\cosh(\theta_t)=\sinh(\xi_t)\left(t\cdot i(\lambda^+,\gamma)-\left|\theta_t\right|\right)
\]
where $\xi_t\in[\left|\theta_t\right|,t\cdot i(\lambda^+,\gamma)]$, and
\[
\cosh(L_t)-\cosh(L_Z(\gamma))=\sinh(\zeta_t)\left(L_t-L_Z(\gamma)\right)
\]
where $\zeta_t\in[\ell_Z(\gamma),L_t]$.

We now divide both right and left hand side by $t^2$ as follows 
\[
\frac{\sinh(\xi_t)}{t}\left(i(\lambda^+,\gamma)-\frac{\left|\theta_t\right|}{t}\right)\le\sinh(\zeta_t)\frac{L_t-\ell_Z(\gamma)}{t^2}
\]
and we observe that as $t\to 0$ the terms converge to: In the left hand side,
\begin{itemize}
\item{$\left|\theta_t\right|/t=\left|L_\gamma(t)-L_\gamma(-t)\right|/2t\to{\dot L}_\gamma$.}
\item{$\sinh(\xi_t)/t\ge\sinh(\left|\theta_t\right|)/t$ as $\xi_t\ge\left|\theta_t\right|$.}
\item{$\sinh(\left|\theta_t\right|)/t\to\cosh(\theta_0){\dot\theta}_0={\dot L}_\gamma$.}
\end{itemize}

In the right hand side,
\begin{itemize}
\item{$\sinh(\zeta_t)\to\sinh(L_\gamma(Z))$ as $L_t=(L_\gamma(t)+L_\gamma(-t))/2\to L_Z(\gamma)$.}
\item{$(L_t-L_Z(\gamma))/t^2=(L_\gamma(t)+L_\gamma(-t)-2L_Z(\gamma))/2t^2\to{\ddot L}_\gamma/2$.}
\end{itemize}
The conclusion follows.

\end{proof}

Let us conclude by recalling that an exact formula for the first variation of length functions along earthquakes has been computed by Kerckhoff \cite{K83}:

\begin{thm}[Kerckhoff \cite{K83}]
\label{thm:kerckhoff}
Let $\lambda\in\mc{ML}$ be a measured lamination. Let $E_\lambda:[a,b]\to\T$ be an earthquake path driven by $\lambda$. Let $\gamma\in\Gamma-\{1\}$ be a non-trivial loop. Set $L_\gamma(t):=L_\gamma(E_\lambda(t))$. We have: 
\[
{\dot L}_\gamma=\int_{\mc{I}/\rho_{E_\lambda(t)}(\Gamma)}{\cos(\theta)\;{ \rm d}\lambda\times\delta_\gamma}
\]
where $\mc{I}$ is the space of intersecting geodesics of $\mb{H}^2$ and $\theta(\ell,\ell')$ is the angle of intersection between the geodesics $\ell,\ell'$.
\end{thm}

In particular, we always have $|{\dot L}_\gamma|\le i(\gamma,\lambda)$ with strict inequality if $\gamma$ intersects $\lambda$ essentially.


\begin{thebibliography}{99}
\bibliographystyle{amsalpha}

\bibitem{BIW10}
M. Burger, A. Iozzi, and A. Wienhard, {\em Surface group representations with maximal Toledo invariant}, Ann. of Math. {\bf 172}(2010), 517-566.

\bibitem{BB09}
R. Benedetti and F. Bonsante, {\em Canonical Wick rotations in 3-dimensional gravity}, Mem. Amer. Math. Soc. {\bf 198}(2009), viii+164. 

\bibitem{BG}
R. Benedetti and E. Guadagnini, {\em The cosmological time in $(2+1)$-gravity}, Nucl. Phys. B {\bf 613}(2001), 330-352. 

\bibitem{BBFS13}
M. Bestvina, K. Bromberg, K. Fujiwara, and J. Souto, {\em Shearing coordinates and convexity of length functions on Teichmüller space}, Amer. J. Math. {\bf 135}(2013), 1449-1476.

\bibitem{Bo88}
F. Bonahon, {\em The geometry of Teichmüller space via geodesic currents}, Invent. Math. {\bf 92}(1988), 139-162. 

\bibitem{Bo96}
F. Bonahon, {\em Shearing hyperbolic surfaces, bending pleated surfaces and Thurston’s symplectic form}, Ann. Fac. Sci. Toulouse {\bf 5}(1996), 233-297.

\bibitem{Bo98}
F. Bonahon, {\em Variations of the boundary geometry of 3-dimensional hyperbolic convex cores}, J. Diff. Geom. {\bf 50}(1998), 1-24.

\bibitem{BS01}
F. Bonahon and Y. Sözen, {\em The Weil-Petersson and Thurston symplectic forms}, Duke Math. J. {\bf 108}(2001), 581-597.

\bibitem{BS20}
F. Bonsante and A. Seppi, {\em Anti-de Sitter geometry and Teichmüller theory}, In the tradition of Thurston, Springer Verlag, 2020.

\bibitem{CEG}
R. Canary, D. Epstein, and B. Green, {\em Notes on notes by Thurston}, London Math. Soc. Lecture Notes 138, Cambridge University Press, Cambridge 2006.

\bibitem{D14}
J. Danciger, {\em Ideal triangulations and geometric transitions}, J. Topol. {\bf 7}(2014), 1118-1154.

\bibitem{G80}
W. Goldman, {\em Discontinuous groups and the Euler class}, Thesis (Ph.D.) University of California, Berkeley, 1980.

\bibitem{K83}
S. Kerckhoff, {\em The Nielsen realization problem}, Ann. Math. {\bf 117}(1983), 235-265.

\bibitem{MV22a}
F. Mazzoli and G. Viaggi, {\em $\mathrm{SO}_0(2,n+1)$-maximal representations and hyperbolic surfaces}, preprint, arxiv:2206.06946.

\bibitem{M07}
G. Mess, {\em Lorentz spacetimes of constant curvature}, Geom. Dedicata {\bf 126}(2007), 3-45.

\bibitem{M93}
Y. Minsky, {\em Teichmüller geodesics and ends of hyperbolic 3-manifolds}, Topology {\bf 32}(1993), 625-647.

\bibitem{The14}
G. Théret, {\em Convexity of length functions and Thurston's shear coordinates}, preprint, arXiv:1408.5771. 

\bibitem{ThNotes}
W. Thurston, {\em Geometry and topology of three-manifolds}, Princeton lecture notes, 1979.

\bibitem{T86}
W. Thurston, {\em Minimal stretch maps between hyperbolic surfaces}, preprint, arXiv:math/9801039.

\bibitem{W68}
F. Waldhausen, {\em On irreducible 3-manifolds which are sufficiently large}, Ann. Math. {\bf 87}(1968), 56-88.

\bibitem{W86}
S. Wolpert, {\em Thurston’s Riemannian metric for Teichmüller space}, J. Diff. Geom. {\bf 23}(1986), pp. 143-174.

\end{thebibliography}
\end{document}